\include{BoxedEPS}
\include{BoxedEPS}
\documentclass[12pt]{amsart}
\usepackage{amsmath,amsthm}
\input amssym.def
\input amssym
\chardef\bslash=`\\ % p. 424, TeXbook
\makeatletter
\def\verbatim{\interlinepenalty\@M \@verbatim
   \leftskip\@totalleftmargin\advance\leftskip2pc
   \frenchspacing\@vobeyspaces \@xverbatim}
\makeatother
\newtheorem{thm}{Theorem}[section]
\newtheorem{coro}[thm]{Corollary}
\newtheorem{lem}[thm]{Lemma}
\newtheorem{prop}[thm]{Proposition}

\theoremstyle{definition}
\newtheorem{defi}{Definition}[section]

\theoremstyle{remark}
\newtheorem{obs}{Remark}[section]

\numberwithin{equation}{section}

\newcommand{\begeq}{\begin {equation}}
\newcommand{\eq}{\end{equation}}
\newcommand{\bs}{\begin {split}}
\newcommand{\es}{\end{split}}
\newcommand{\bp}{\begin {prop}}
\newcommand{\ep}{\end {prop}}
\newcommand{\bt}{\begin {thm}}
\newcommand{\et}{\end {thm}}
\newcommand{\bc}{\begin {cor}}
\newcommand{\ec}{\end {cor}}
\newcommand{\bl}{\begin {lem}}
\newcommand{\el}{\end {lem}}
\newcommand{\bpf}{\begin {proof}}
\newcommand{\epf}{\end {proof}}
\newcommand{\bi}{\begin {itemize}}
\newcommand{\ei}{\end {itemize}}
\newcommand{\ben}{\begin {enumerate}}
\newcommand{\een}{\end {enumerate}}
\newcommand{\brem}{\begin {rem}}
\newcommand{\erem}{\end {rem}}

\newcommand{\norm}[1]{\left\|{#1} \right\|}

\newcommand{\la}{\langle}
\newcommand{\ra}{\rangle}

\newcommand{\mathds}{\mathbb}
\newcommand{\HH}{{\mathcal H}}
\newcommand{\ZZ}{{\mathbb Z}}

\newcommand{\RR}{{\mathbb R}}

\newcommand{\NN}{{\mathbb N}}
\newcommand{\GG}{\ZZ^d}

\newcommand{\lp}{{\ell^p} }

\newcommand{\lt}{{\ell^2} }

\numberwithin{equation}{section}

\begin{document}
\title{\bf On Stability of Sampling-Reconstruction Models}

\author{Ernesto Acosta-Reyes, Akram Aldroubi, and Ilya Krishtal}
\address{Dept. of Mathematics, Vanderbilt University, Nashville, TN 37240 \\
email: ernesto.acosta@vanderbilt.edu}
\address{Dept. of Mathematics, Vanderbilt University, Nashville, TN 37240 \\
email: akram.aldroubi@vanderbilt.edu}
\address{Dept. of Mathematical Sciences, Northern Illinois University, DeKalb, IL 60115 \\
email: krishtal@math.niu.edu}

\footnotetext{\textit{Math Subject Classifications.}
                       42C15.}
\thanks{ The second author was supported in part by NSF grants DMS-0504788.}

\maketitle

\begin{abstract}
A useful sampling-reconstruction model should be stable with respect to different kind
of small perturbations, regardless whether they result from jitter,
measurement errors, or simply from a small change in the model assumptions.
In this paper we prove this result for a large class of sampling models.
We define different classes of perturbations and quantify the robustness of
a model with respect to them. We also use the theory of localized frames
to study the frame algorithm for recovering the original signal from its samples.

\end{abstract}

\section{Introduction}

The sampling and reconstruction problem includes devising efficient methods for
representing a signal (function) in terms of a discrete (finite or countable) set of its samples (values) and
reconstructing the original signal from the samples (see e.g., \cite {aa,aagr1, BF01,BS01, NW91,U00} and the reference therein). In this paper we consider a very general
sampling model where the signal is assumed to belong to a finitely generated shift invariant space
and the sampling is performed on an irregular separated set and is averaged by finite Borel measures.
The main focus of this paper is on describing and quantifying ``admissible'' perturbations of the sampling model
which may result from altering the sampling set (jitter) (see e.g. \cite {aacl,ABK02,FS06}),  or the averaging sampling measures (measuring devices)
or the generators of the underlying shift-invariant space (see e.g., \cite {aaik,SUN06}).

As recently became customary in sampling theory (see e.g. \cite{aa,aagr1,CIS02,SZ03,SZ03b,S05,U00}), we mesh operator theory techniques and those of shift invariant
and Wiener amalgam spaces \cite{hgf}. The latter provide us with relatively straight-forward proofs %and sharper estimates
while the former allow us to keep in sight our objective. In section 2 we show that all the properties of
our sampling model can be encoded in the sampling operator $U$. The sampling model admits reconstruction
if its sampling operator is bounded both above and below. Our first goal is to show that any and all of the small perturbations
mentioned above result in a small perturbation of $U$ in the operator norm.
This will prove the stability of sampling in our model with respect to those perturbations and  the corresponding estimates we obtain
will quantify this stability. Our second goal is to show how a frame algorithm can be used to reconstruct signals
in our sampling model. Finally, our last goal is to show that the reconstruction error due to the perturbations we describe
is controlled continuously by the perturbation errors.

The paper is organized as follows. In section 2 we describe our sampling model,
introduce relevant notions and notation, and cite a few preliminary results.
 The main results are presented in section 3.
Perturbation results %with a fixed sampling set
addressing our first goal
are in subsection 3.1. There we prove that a set of
sampling remains such under a small perturbation of
the sampling measures and/or the generators of the shift invariant space. It is also shown
that sampling remains stable with respect to a perturbation of the sampling set itself.
In subsection 3.2 we show that, in case of a signal in a Hilbert space,
a frame algorithm can be used to reconstruct the function from its samples. We also
use the results of the previous subsection and the theory of localized frames
to show that under mild additional assumptions
a set of sampling for a Hilbert shift invariant space is also
a set of sampling for a chain of Banach shift invariant spaces to which the frame algorithm extends.
In subsection 3.3 we study the dependence of the reconstruction error upon the perturbation errors. The proofs of
the results in section 3 are relegated to section 4.

%__________________________________________________________________________________________________________________________
%Section2:
%\section{Notation and preliminaries}

 \section{Description of the sampling model}

This section is primarily devoted to introduction of the sampling model we use in this paper. We also present most of the necessary notation and cite
some of the preliminary results that will be used later.

The signals we are studying in this paper are represented by functions $f\in L^p(\mathbb{R}^d)$,
for some $p\in[1,\infty]$ and $d\in\NN$. Moreover, we assume that $f$ belongs to
a shift invariant space

\begin{equation}\label{1.1}
V^{p}(\Phi)=\{\sum_{k\in \mathbb{Z}^{d}}C_{k}^{T}\Phi_{k}:
C\in(\ell^p(\mathbb{Z}^d))^{(r)}\}.
\end{equation}
Here
$\Phi=(\phi^{1},\ldots,\phi^{r})^{T}$ is a vector of functions,
$\Phi_{k}=\Phi(\cdot-k)$, and $C=(c^{1},\ldots,c^{r})^{T}$ is a
vector of sequences belonging to $(\ell^p(\mathds{Z}^d))^{(r)}$. Among
the equivalent norms in $(\ell^p(\mathds{Z}^d))^{(r)}$ we choose
\begin{displaymath}
\label {ellpnorm}
\|C\|_{(\ell^p(\mathds{Z}^d))^{(r)}}=\sum_{i=1}^{r}\|c^{i}\|_{\ell^p(\mathds{Z}^d)}.
\end{displaymath}
In order to avoid convergence issues in \eqref{1.1} we assume that the
set $\{\phi^{1}(\cdot-k),\ldots,\phi^{r}(\cdot-k);
k\in\mathbb{Z}^d\}$ generates an unconditional basis for
$V^{p}(\Phi)$. In particular, we require that there exist
constants $0<m_{p}\leq M_{p}<\infty,$ such that %, for all $,

\begin{equation}\label{2.1.1}
m_{p}\|C\|_{(\ell^{p}(\mathbb{Z}^d))^{(r)}} \leq
\|\sum_{k\in\mathbb{Z}^d}C_{k}^{T}\Phi_{k}\|_{L^p} \leq
M_{p}\|C\|_{(\ell^{p}(\mathbb{Z}^d))^{(r)}},\ \forall C \in
(\ell^{p}(\mathbb{Z}^d))^{(r)}.
\end{equation}

The unconditional basis assumption (\ref{2.1.1}) implies \cite{aagr1} that the
space $V^{p}(\Phi)$ is a closed subspace of $L^p(\mathbb{R}^d).$

Since we are interested in sampling in $V^p(\Phi)$ we add
an assumption that would make all the functions in these spaces continuous and,
therefore, pointwise evaluations will be meaningful. To this end, we assume that
all generators $\Phi$ belong to a Wiener-amalgam
space $(\mathit{W}_{0}^{1})^{(r)}$ as defined below. For $1 \leq p
<\infty$, a measurable function $f$ belongs to $\mathit{W}^{p}$ if
it satisfies
\begin{eqnarray}\label{2.1.2}
\|f\|_{\mathit{W}^p}=\left(\sum_{k\in\mathbb{Z}^d}
\underset{{x\in[0,1]^{d}}}{\operatorname{esssup}}\,|f(x+k)|^{p}\right)^{1/p}<\infty.
\end{eqnarray}
If $p=\infty$, a measurable function $f$ belongs to
$\mathit{W}^{\infty}$ if it satisfies
\begin{eqnarray}\label{2.1.3}
\|f\|_{\mathit{W}^{\infty}}=\sup_{k\in\mathbb{Z}^d}\{\underset{{x\in[0,1]^{d}}}{\operatorname{esssup}}\,|f(x+k)|\}<\infty.
\end{eqnarray}
Hence, $\mathit{W}^{\infty}$ coincides with
$L^{\infty}(\mathbb{R}^d)$. It is well known that $\mathit{W}^{p}$
are Banach spaces \cite{hgf}, and clearly $\mathit{W}^{p} \subseteq
L^p$. By $(\mathit{W}^{p})^{(r)}$ we denote the space of
vectors $\Psi=(\psi^{1},\ldots,\psi^{r})^{T}$ of
$\mathit{W}^{p}$-functions with the norm
\begin{displaymath}
\|\Psi\|_{(\mathit{W}^{p})^{(r)}}=\sum_{i=1}^{r}\|\psi^{i}\|_{\mathit{W}^{p}}.
\end{displaymath}
The closed subspace of (vectors of) continuous functions in
$\mathit{W}^{p}$ (respectively, $(\mathit{W}^{p})^{(r)}$) will be
denoted by $\mathit{W}_{0}^{p}$ (or ($\mathit{W}_{0}^{p})^{(r)}$).

In this paper we are interested in average
sampling performed by a vector of measures. We denote by  $\mathcal{M}(\mathbb{R}^d)=\mathcal{M}_0(\mathbb{R}^d)$
the Banach space of finite complex Borel measures on $\mathbb{R}^d$. The norm on
$\mathcal{M}(\mathbb{R}^d)$ is given by $\|\mu\|=\int_{\mathbb{R}^d}d|\mu|(y)$, i.e., the total variation of
a measure $\mu$.
By $(\mathcal{M}(\mathbb{R}^d))^{(t)}$ we denote the space of vectors
$\overrightarrow{\mu}=(\mu^{1},\ldots,\mu^{t})$ of measures from $\mathcal M(\mathbb{R}^d)$ with the norm
$\|\overrightarrow{\mu}\|_{(\mathcal{M}(\mathds{R}^d))^{(t)}}=\sum_{j=1}^{t}\|\mu^{j}\|$.
The symbols $\mathcal{M}_{s}(\mathbb{R}^d)$
($(\mathcal{M}_{s}(\mathds{R}^d))^{(t)}$), $0\le s<\infty$, will be used for the subspace of
$\mathcal{M}(\mathbb{R}^d)$ ($(\mathcal{M}(\mathbb{R}^d))^{(t)}$) of
all (vectors of) measures $\mu\in \mathcal{M}(\mathbb{R}^d)$ such that $(1+|x|)^s \in L^1(\RR^d, d|\mu|)$, i.e.,
$\int (1+|x|)^s d|\mu|(x) < \infty$. By
$\mathcal{M}_{\infty}(\mathbb{R}^d)$ ($(\mathcal{M}_{\infty}(\mathbb{R}^d))^{(t)}$) we denote the space of
all (vectors of) measures with compact support.
Clearly $\mathcal{M}_{s}(\mathbb{R}^d)\subset \mathcal{M}_{r}(\mathbb{R}^d)$ for $0\le r\le s \le\infty$.

 For $\mu \in
\mathcal{M}(\mathbb{R}^d)$ and a measurable function $\phi$ on
$\mathbb{R}^d$,  the convolution of the function $\phi$ and
the measure $\mu$ is defined by
\begin{displaymath}
(\phi\ast\mu)(x)=\int_{\mathbb{R}^d}\phi(x-y)d\mu(y), %\quad\textrm{for
%all}\quad
\quad x\in\mathbb{R}^d.
\end{displaymath}
When we have a vector of measurable functions
$\Phi=(\phi^1,\ldots,\phi^r)^{T}$ and a vector of finite complex Borel
measures $\overrightarrow{\mu}=(\mu^{1},\ldots,\mu^{t})$, then the
convolution $\Phi\ast\overrightarrow{\mu}$ is the $r\times t$ matrix
given by
\begin{displaymath}
\Phi\ast\overrightarrow{\mu}=\left(\begin{array}{ccc}
\phi^{1}\ast\mu^{1}&
\ldots&\phi^{1}\ast\mu^{t}\\
\vdots&&\vdots\\
\phi^{r}\ast\mu^{1}&\ldots&\phi^{r}\ast\mu^{t}
\end{array}\right).
\end{displaymath}
%Let $X=\{x_{j}: j \in J\}$ be a countable subset of
%$\mathbb{R}^d$, then we have the following definition.

Let $J$ be a countable index set
and $X=\{x_{j}: j \in J\}$ be a subset of $\mathbb{R}^d$.
The reconstruction problem in our sampling model consists of finding the
function $f\in V^p(\Phi)$ from the knowledge of its samples
\begin{displaymath}
%\{s_{x_{j}}(f)
(f\ast\overrightarrow{\mu})(X)=\{(f\ast\overrightarrow{\mu})(x_{j})=\left((f\ast\mu^{1})(x_{j}),\ldots,(f\ast\mu^{t})(x_{j})\right)\}_{j
\in J}.
\end{displaymath}

When $t=1$ and $\mu =\delta_{0}$, i.e., $\mu$
is the Dirac measure on $\mathbb{R}^d$ concentrated at zero, then
$(f\ast\overrightarrow{\mu})(X)= \{f(x_{j})\}_{j \in J}$
and we obtain the classical (ideal) sampling model.
When $d\overrightarrow{\mu}=\Psi dx$, where
$\Psi\in(L^1(\mathbb{R}^d))^{(t)}$ and $dx$ is the Lebesgue measure
on $\mathbb{R}^d$, i.e., $\overrightarrow{\mu}$ is absolutely
continuous with respect to the Lebesgue measure, then we write
$(f\ast\Psi)(X)$
instead of $(f\ast\overrightarrow{\mu})(X)$, and  our model is reduced to
the case analyzed in \cite{aaik}.

%These spaces are called shift-invariant.

\begin{defi}\label{def1.2} Let $1 \leq p \leq \infty$ and $X=\{x_{j}: j \in J\}$ be a countable subset of
$\mathbb{R}^{d}$. We say that $X$ is a \emph{set of sampling} for
$V^p(\Phi)$ and $\overrightarrow{\mu}$ (or, simply, a $\overrightarrow{\mu}$-sampling set for $V^p(\Phi)$)
if there exist constants
$0<A_{p}\leq B_{p}<\infty$ such that
\begin{equation}\label{1.2}
A_{p}\|f\|_{L^{p}}\leq
\|(f\ast\overrightarrow{\mu})(X)\|_{(\ell^{p}(J))^{(t)}}\leq
B_{p}\|f\|_{L^{p}},\mbox{  for all } f\in V^{p}(\Phi).
\end{equation}
\end{defi}

If $d\overrightarrow{\mu}=\Psi dx$ then a $\overrightarrow{\mu}$-sampling set $X$ will be called a
$\Psi$-sampling set and, if $t=1$ and $\mu =\delta_{0}$, then $X$ will be called an \emph{ideal} sampling set.
To ensure that an upper bound $B_p$ in \eqref{1.2} always exists (see \eqref{4.1.2}) we restrict our attention
only to separated sets $X$.

\begin{defi}\label{def1.1} We say
that $X$ is \emph{separated} if there exists $\delta>0$ such that
$\inf_{i,j \in J, i\neq j}|x_{i}-x_{j}|\geq\delta$. The number $\delta$ is
called the \emph{separation constant} of the set $X$.
\end{defi}

It is not hard to extend our results to the case of a finite union of separated sets.
 We do not, however, pursue this relatively trivial but space consuming generalization.

%____________________________________________________________

\begin{defi}\label{defSampMod}
Let $\overrightarrow{\mu}\in(\mathcal{M}(\mathbb{R}^d))^{(t)}$,
$\Phi\in(\mathit{W}_{0}^{1})^{(r)}$ satisfy (\ref{2.1.1}), and
$X=\{x_{j},j\in J\}\subset\mathbb{R}^d$ be a separated set.
The \emph{sampling model} is the triple $(X,\Phi,\overrightarrow{\mu})$.
The sampling model $(X,\Phi,\overrightarrow{\mu})$ is called \emph{$p$-stable} if $X$ is a
$\overrightarrow{\mu}$-sampling set for $V^p(\Phi)$, $p\in[1,\infty]$.
\end{defi}

Given a sampling model $(X,\Phi,\overrightarrow{\mu})$ we proceed to define its sampling operator.

\begin{defi}\label{defSampOp}
The \emph{sampling operator}  $U = U_{(X,\Phi,\overrightarrow{\mu})}:(\ell^p(\mathbb{Z}^d))^{(r)}\to (\ell^p(J))^{(t)}$ %that corresponds to the set $X$. Let $U$ be
%the linear operator on $(\ell^p(\mathbb{Z}^d))^{(r)}$ so that
is defined by $UC=(f\ast\overrightarrow{\mu})(X)$, where $f = \sum\limits_{k\in \mathbb{Z}^{d}}C_{k}^{T}\Phi_{k}\in V^p(\Phi)$.
\end{defi}
We can think of $U$ as a
$t\times r$ matrix of operators %given by
\begeq\nonumber
U =\left(
\begin{array}{ccc}
U^{1,1} & \ldots & U^{r,1}\\
\vdots & &\vdots\\
U^{1,t}&\ldots&U^{r,t}
\end{array}\right),
\end{equation}
where for each $1\leq i\leq r$ and $1\leq l\leq t$ the operator
$U^{i,l}$ is defined by an infinite matrix %whose $j$, $k$
with entries $(U^{i,l})_{j,k}=(\phi^{i}\ast\mu^{l})(x_{j}-k)$,
$j\in J$, $k\in\mathbb{Z}^d$. The operator norm of  $U$
is given by $\|U\|_{p,op}=\sum_{l=1}^{t}\sum_{i=1}^{r}\|U^{i,l}\|$.

The following proposition shows that all the interesting properties of a sampling model $(X,\Phi,\overrightarrow{\mu})$
are, indeed, encoded in the sampling operator $U$. The proof of this result follows immediately from \eqref{2.1.1} and \eqref{1.2}.
\begin{prop}\label{prop5.1}
The sampling model $(X,\Phi,\overrightarrow{\mu})$ is $p$-stable
%Assume that (\ref{2.1.1}) holds for
%all $C\in(\ell^p(\mathbb{Z}^d))^{(r)}$. Then $X$ is a sampling set
%for $V^p(\Phi)$ and $\overrightarrow{\mu}$
if and only if there exist $0<\eta_{p}\leq\beta_{p}<\infty$ such
that for all $C\in(\ell^p(\mathbb{Z}^d))^{(r)}$ the sampling
operator $U$ satisfies
\begin{equation}\label{5.1}
\eta_{p}\|C\|_{(\ell^p(\mathbb{Z}^d))^{(r)}}\leq\|UC\|_{(\ell^p(J))^{(t)}}\leq\beta_{p}\|C\|_{(\ell^p(\mathbb{Z}^d))^{(r)}}.
\end{equation}
\end{prop}

The next lemma is, essentially, a nutshell for many of the results in this paper.

%Lemma5.1
\begin{lem}\label{nutshell} Let $(X,\Phi,\overrightarrow{\mu})$ be a $p$-stable sampling model
and $U$ be its sampling operator satisfying (\ref{5.1}).
Let also $(\widetilde{X},\Theta,\overrightarrow{\alpha})$ be a sampling model
such that its sampling operator $U_{\Delta}$ satisfies
$\|U-U_{\Delta}\|<\eta_{p}$. Then $(\widetilde{X},\Theta,\overrightarrow{\alpha})$ is also $p$-stable.
%$X+\Delta$ is also a set of sampling for $V^{p}(\Phi)$ and $\overrightarrow{\mu}$.
\end{lem}
\begin{proof}
Let $C\in(\ell^p(\mathbb{Z}^d))^{(r)}$. Then
\begin{eqnarray}
\|U_{\Delta}C\|_{(\ell^p(J))^{(t)}}&\leq&\|(U-U_{\Delta})C\|_{(\ell^p(J))^{(t)}}+\|UC\|_{(\ell^p(J))^{(t)}}\nonumber\\
&\leq&\|U-U_{\Delta}\|\|C\|_{(\ell^p(\mathbb{Z}^d))^{(r)}}+\beta_{p}\|C\|_{(\ell^p(\mathbb{Z}^d))^{(r)}}.\nonumber
\end{eqnarray}
Therefore, since $\|U-U_{\Delta}\|<\eta_{p}$, then we have
\begin{equation}\label{5.1.111}
\|U_{\Delta}C\|_{(\ell^p(J))^{(t)}}\leq\left(\eta_{p}+\beta_{p}\right)\|C\|_{(\ell^p(\mathbb{Z}^d))^{(r)}}.
\end{equation}
On the other hand, since
\begin{eqnarray}
\eta_{p}\|C\|_{(\ell^p(\mathbb{Z}^d))^{(r)}}&\leq&\|UC\|_{(\ell^p(J))^{(t)}}\leq\|(U-U_{\Delta})C\|_{(\ell^p(J))^{(t)}}+\|U_{\Delta}C\|_{(\ell^p(J))^{(t)}}\nonumber\\
&\leq&\|U-U_{\Delta}\|\|C\|_{(\ell^p(\mathbb{Z}^d))^{(r)}}+\|U_{\Delta}C\|_{(\ell^p(J))^{(t)}}.\nonumber
\end{eqnarray}
Hence,
\begin{equation}\label{5.1.211}
\left(\eta_{p}-\|U-U_{\Delta}\|\right)\|C\|_{(\ell^p(\mathbb{Z}^d))^{(r)}}\leq\|U_{\Delta}C\|_{(\ell^p(J))^{(t)}}.
\end{equation}
Since $\|U-U_{\Delta}\|<\eta_{p}$, the conclusion of the
lemma follows from (\ref{5.1.111}),
(\ref{5.1.211}), and Proposition \ref{prop5.1}.
\end{proof}

%________________________________________________________________________________________________________________________
%Section3:
\section{Main Results}
In this section we collect the main results of our paper.

\subsection{Admissible perturbations of a sampling model.}\

In practice, shift invariant spaces are used to model  classes of signals that can occur (or that are allowed) in applications. However often, the functions in a shift invariant space model only give approximations to the signals of interest.   For this reason, we begin with a result where the perturbation of a sampling model is
due to a small change of the genetators of the underlying shift invariant space.
%Theorem3.1
\begin{thm}\label{teo3.1}
Let $(X,\Phi,\overrightarrow{\mu})$ be a $p$-stable sampling model for some $p\in[1,\infty]$.
%Let $\overrightarrow{\mu}\in(\mathcal{M}(\mathbb{R}^d))^{(t)}$,
%$\Phi\in(\mathit{W}_{0}^{1})^{(r)}$ satisfy (\ref{2.1.1}), and
%$X=\{x_{j},j\in J\}\subset\mathbb{R}^d$ be a separated $\overrightarrow{\mu}$-sampling set for $V^p(\Phi)$.
Then there exists $\epsilon_0 > 0$ such that
%$X$ is a $\overrightarrow{\mu}$-sampling set for $V^p(\Theta)$,
the sampling model $(X,\Theta,\overrightarrow{\mu})$ is also $p$-stable,
whenever $\Theta\in(\mathit{W}_{0}^{1})^{(r)}$ and
$\|\Phi-\Theta\|_{(\mathit{W}^{1})^{(r)}}<\epsilon_0$.
\end{thm}

The above result means that if
$\overrightarrow{\mu}\in(\mathcal{M}(\mathbb{R}^d))^{(t)}$,
$\Phi\in(\mathit{W}_{0}^{1})^{(r)}$ satisfies (\ref{2.1.1}),
$X=\{x_{j},j\in J\}\subset\mathbb{R}^d$ is a separated $\overrightarrow{\mu}$-sampling set for $V^p(\Phi)$, and $\Theta$  satisfies the assumptions of the theorem, then
there exist $0< A^{'}_{p} \leq B^{'}_{p}<\infty$ such that
\begin{equation}\label{3.1.2}
A^{'}_{p}\|g\|_{L^p}\leq\|(g\ast\overrightarrow{\mu})(X)\|_{(\ell^{p}(J))^{(t)}}\leq
B^{'}_{p}\|g\|_{L^p},\mbox{  for all }  g\in V^{p}(\Theta).
\end{equation}
In the proof of this result in section 4 we will provide explicit
estimates for $\epsilon_0$ and the bounds $A^{'}_{p}$ and
$B^{'}_{p}$.

As a consequence of Theorem \ref{teo3.1} we have the following
results that were first proved in \cite{aaik}. The proofs now are
immediate: we apply Theorem \ref{teo3.1} with $d\overrightarrow{\mu}
= \Psi dx$ for Corollary \ref{coro3.4} and $\overrightarrow{\mu} =
\delta_0$ for Corollary \ref{coro3.5}.

%Corollaries to Theorem 3.1

%Corollary 3.4
\begin{coro}\label{coro3.4}
Let $\Psi\in(L^1(\mathbb{R}^d))^{(t)}$,
$\Phi\in(\mathit{W}_{0}^{1})^{(r)}$ satisfy (\ref{2.1.1}), and
$X=\{x_{j},j\in J\}\subset\mathbb{R}^d$ be a separated $\Psi$-sampling set for $V^p(\Phi)$.
Then there exists $\epsilon_0 > 0$ such that
$X$ is a $\Psi$-sampling set for $V^p(\Theta)$,
whenever $\Theta\in(\mathit{W}_{0}^{1})^{(r)}$ and
$\|\Phi-\Theta\|_{(\mathit{W}^{1})^{(r)}}<\epsilon_0$.
\end{coro}
%Corollary 3.5
\begin{coro}\label{coro3.5}
Let $\Phi\in(\mathit{W}_{0}^{1})^{(r)}$ satisfying (\ref{2.1.1})
and $X=\{x_{j},j\in J\}\subset\mathbb{R}^d$ be a separated ideal set of sampling for $V^{p}(\Phi)$.
Then there exists $\epsilon_{0}>0$ such that
$X$ is an ideal set of sampling for $V^{p}(\Theta)$,
whenever $\Theta\in(\mathit{W}_{0}^{1})^{(r)}$ and
$\|\Phi-\Theta\|_{(\mathit{W}^{1})^{(r)}}\leq\epsilon<\epsilon_{0}$.
\end{coro}

In practice, signal samples are obtained using measuring devices with characteristics that are not fully known, and the measurements reflect local averages rather than  exact sample values. Thus,  a sampling measure $\overrightarrow{\mu}$ is a model that approximate the characterisitics of a measuring device.  For this reason, the next theorem describes the case when the perturbation is due to some uncertainty about the characteristics of the measuring devices, that is a perturbation of the vector of measures $\overrightarrow{\mu}$.

%Theorem 3.3.
\begin{thm}\label{teo3.3}
Let $(X,\Phi,\overrightarrow{\mu})$ be a $p$-stable sampling model for some $p\in[1,\infty]$.
Then there exists $\epsilon_0 > 0$ such that
the sampling model $(X,\Phi,\overrightarrow{\alpha})$ is also $p$-stable,
%Let $\overrightarrow{\mu}\in(\mathcal{M}(\mathbb{R}^d))^{(t)}$,
%$\Phi\in(\mathit{W}_{0}^{1})^{(r)}$ satisfy (\ref{2.1.1}), and
%$X=\{x_{j},j\in J\}\subset\mathbb{R}^d$ be a separated $\overrightarrow{\mu}$-sampling set for $V^p(\Phi)$.
%Then there exists $\epsilon_0 > 0$ such that
%$X$ is an $\overrightarrow{\alpha}$-sampling set for $V^p(\Phi)$,
whenever
$\overrightarrow{\alpha}\in(\mathcal{M}(\mathbb{R}^d))^{(t)}$ and
\[\|\overrightarrow{\mu}-\overrightarrow{\alpha}\|_{(\mathcal{M}(\mathbb{R}^d))^{(t)}}<\epsilon_{0}.\]
\end{thm}

Again, if $X$, $\overrightarrow{\mu}$, $\overrightarrow{\alpha}$, and $\Phi$ satisfy the assumptions of the theorem then
there exist $0< A^{'}_{p} \leq B^{'}_{p}<\infty$ such that
\begin{equation}\label{3.3.1.1}
A^{'}_{p}\|f\|_{L^p}\leq\|(f\ast\overrightarrow{\alpha})(X)\|_{(\ell^{p}(J))^{(t)}}\leq
B^{'}_{p}\|f\|_{L^p},\mbox{  for all }  f\in V^{p}(\Phi),
\end{equation}
and the explicit estimates for $\epsilon_0$, $A^{'}_{p}$ and
$B^{'}_{p}$ will be given in section 4.

Considering $\overrightarrow{\mu}$ and $\overrightarrow{\alpha}$ in Theorem \ref{teo3.3}
such that  $d\overrightarrow{\mu} = \Psi dx$ and $d\overrightarrow{\alpha} = \Gamma dx$
 we obtain the following direct corollary (see also \cite[Theorem 3.3]{aaik}).

%Corollary 3.6
\begin{coro}\label{coro3.6}
Let $\Psi\in(L^1(\mathbb{R}^d))^{(t)}$,
$\Phi\in(\mathit{W}_{0}^{1})^{(r)}$ satisfy (\ref{2.1.1}), and
$X$ be a separated $\Psi$-sampling set for $V^p(\Phi)$.
Then there exists $\epsilon_0 > 0$ such that
$X$ is a $\Gamma$-sampling set for $V^p(\Phi)$,
whenever $\Gamma\in(L^1(\mathbb{R}^d))^{(t)}$ and
$\|\Psi-\Gamma\|_{(L^1(\mathbb{R}^d))^{(t)}}<\epsilon_{0}$.
\end{coro}

As a consequence of Theorems \ref{teo3.1} and \ref{teo3.3} we
obtain the following combined perturbation result and its
corollary, which is essentially Theorem 3.4 in \cite{aaik}.

%Theorem 3.4.
\begin{thm}\label{teo3.4}
Let $(X,\Phi,\overrightarrow{\mu})$ be a $p$-stable sampling model for some $p\in[1,\infty]$.
Then there exists $\epsilon_0 > 0$ such that
the sampling model $(X,\Theta,\overrightarrow{\alpha})$ is also $p$-stable,
whenever
$\overrightarrow{\alpha}\in(\mathcal{M}(\mathbb{R}^d))^{(t)}$,
$\Theta\in(\mathit{W}_{0}^{1})^{(r)}$, and
$\|\Phi-\Theta\|_{(\mathit{W}^{1})^{(r)}}+\|\overrightarrow{\mu}-\overrightarrow{\alpha}\|_{(\mathcal{M}(\mathbb{R}^d))^{(t)}}<\epsilon_{0}$.
\end{thm}

%From theorem \ref{teo3.4}, setting $\overrightarrow{\mu} = \Psi dx$ and $\overrightarrow{\alpha} = \Gamma dx$,
% we obtain:

%Corollary 3.7
\begin{coro}\label{coro3.7}
Let $\Psi\in(L^1(\mathbb{R}^d))^{(t)}$,
$\Phi\in(\mathit{W}_{0}^{1})^{(r)}$ satisfy (\ref{2.1.1}), and
$X$ be a separated $\Psi$-sampling set for $V^p(\Phi)$.
Then there exists $\epsilon_0 > 0$ such that
$X$ is a $\Gamma$-sampling set for $V^p(\Theta)$,
whenever $\Gamma\in(L^1(\mathbb{R}^d))^{(t)}$,
$\Theta\in(\mathit{W}_{0}^{1})^{(r)}$ and
$\|\Phi-\Theta\|_{(\mathit{W}^{1})^{(r)}}+\|\Psi-\Gamma\|_{(L^1(\mathbb{R}^d))^{(t)}}<\epsilon_{0}$.
\end{coro}

An error in the location of the sampling points $\{x_j\}$ is what is
often called jitter error  (see e.g., \cite {aacl,ABK02} and the
references therein). This error can be modeled as a perturbation of
the sampling set $X$. For this reason, our next perturbation results
deal with an altered sampling set $\widetilde{X} = X+\Delta = \{x_j
+\delta_j\}_{j\in J}$, where $\Delta=\{\delta_{j}\}_{j\in
J}\subset\mathbb{R}^d$. We use the standard notation for
$\norm{\Delta}_\infty = \sup\{\norm{\delta_{j}}: j\in J\}$.
 
 \begin{thm}\label{teo5.1}
Let $(X,\Phi,\overrightarrow{\mu})$ be a $p$-stable sampling model for some $p\in[1,\infty]$.
Then there exists $\epsilon_0 > 0$ such that
the sampling model $(X+\Delta,\Phi,\overrightarrow{\mu})$ is also $p$-stable,
whenever
$\norm{\Delta}_\infty <\epsilon_0$.
\end{thm}

\begin{obs}\label{obs4} The above theorem is an analog of Theorem 3.6 in \cite{aacl}, where $r=t=1$,
$p=2$, and $\mu=\mu^1=\delta_{0}$.
\end{obs}

As a direct corollary of Theorems \ref{teo3.4} and \ref{teo5.1} we get the following combined result.

\begin{thm}\label{teo5.111}
Let $(X,\Phi,\overrightarrow{\mu})$ be a $p$-stable sampling model for some $p\in[1,\infty]$.
Then there exists $\epsilon_0 > 0$ such that
the sampling model $(X+\Delta,\Theta,\overrightarrow{\alpha})$ is also $p$-stable,
whenever
$\overrightarrow{\alpha}\in(\mathcal{M}(\mathbb{R}^d))^{(t)}$,
$\Theta\in(\mathit{W}_{0}^{1})^{(r)}$, and
$\norm{\Delta}_\infty+\|\Phi-\Theta\|_{(\mathit{W}^{1})^{(r)}}+\|\overrightarrow{\mu}-\overrightarrow{\alpha}\|_{(\mathcal{M}(\mathbb{R}^d))^{(t)}}<\epsilon_{0}$.
\end{thm}

We leave it to the reader to formulate other perturbation theorems resulting from different combinations of
Theorems \ref{teo3.1}, \ref{teo3.3}, and \ref{teo5.1}. We conclude this section with a slightly stronger version (due to Lemma \ref{nutshell})
of Theorem \ref{teo5.111}.

\bt\label{bigshell}
Let $(X,\Phi,\overrightarrow{\mu})$ be a $p$-stable sampling model for some $p\in[1,\infty]$ and $U$ be its sampling operator.
Let also $(X+\Delta,\Theta,\overrightarrow{\alpha})$ be a perturbed sampling model with the sampling operator $U_\Delta$.
Then for every $\epsilon > 0$ there exists $\epsilon_0 > 0$ such that $\norm{U-U_\Delta} < \epsilon$,
%the sampling model  is also $p$-stable,
whenever
$\overrightarrow{\alpha}\in(\mathcal{M}(\mathbb{R}^d))^{(t)}$,
$\Theta\in(\mathit{W}_{0}^{1})^{(r)}$, and
\[\norm{\Delta}_\infty+\|\Phi-\Theta\|_{(\mathit{W}^{1})^{(r)}}+\|\overrightarrow{\mu}-\overrightarrow{\alpha}\|_{(\mathcal{M}(\mathbb{R}^d))^{(t)}}<\epsilon_{0}.\]
\et

%++++++++++++++++++++++++++++++++++++++++++++++++++++++++++++++++++++++++++++++++++++

\subsection{Perfect reconstruction and localized frames.}\

In this section we show that a frame algorithm can be used to reconstruct $f\in V^2(\Phi)$ from its samples.
We also obtain a useful modification of the above results using
the theory of localized frames developed in \cite{kg1}
(see Definition \ref{def3.2}). In the previous section,
the number $p\in[1,\infty]$ was fixed, that is, we stated,
for example, that if $X$ is a $\overrightarrow{\mu}$-sampling set for $V^p(\Phi)$,
then $X$ is a $\overrightarrow{\mu}$-sampling set for $V^p(\Theta)$ for \emph{the same}
$p\in[1,\infty]$, as soon as $\Theta$ is sufficiently close to $\Phi$
in the appropriate norm. Here, we claim that
if $X$ is a $\overrightarrow{\mu}$-sampling set for $V^2(\Phi)$,
then $X$ is a $\overrightarrow{\mu}$-sampling set for $V^p(\Theta)$ for \emph{all}
$p\in[1,\infty]$, as soon as $\Theta$ is sufficiently close to $\Phi$,
$\Phi$ satisfies a mild decay condition, and $\overrightarrow{\mu}$ belongs to $\mathcal M_s(\RR^d)$ for some $s > d$.
It is natural to ask whether one can replace $V^2(\Phi)$ in the above statement
with $V^q(\Phi)$, \emph{for some} $q\in[1,\infty]$. Under certain
assumptions the answer is ``yes'', but it turns out to be a much harder
problem as shown in \cite{ABK}.

\begin{defi}\label{def3.1} Let $\mathcal{H}$ be a Hilbert space of functions
and $V$ a closed subspace of $\mathcal{H}$. Let
$\{\Psi_{x_{j}}=(\psi_{x_{j}}^1,\ldots,\psi_{x_{j}}^t)^{T}\}_{j\in
J}$ be a countable collection of vectors of functions in $V$.
We say that $\{\Psi_{x_{j}}\}_{j\in J}$ is a frame for $V$ if there
exist constants $0<A\leq B<\infty$ such that
\begin{displaymath}
A\|f\|_\HH\leq \|\langle f,\Psi_{x_{j}}\rangle\|_{(\ell^2(J))^{(t)}}\leq
B\|f\|_\HH, \mbox{ for all } f\in V,
\end{displaymath}
where $\langle f,\Psi_{x_{j}}\rangle=(\langle
f,\psi_{x_{j}}^1\rangle,\ldots,\langle
f,\psi_{x_{j}}^t\rangle)\in\mathbb{C}^t$.
\end{defi}

\begin{obs}\label{obs3.2} Notice that the above is not quite the standard
definition of a frame in a Hilbert space. This is due to the way we
defined the norm in \eqref {ellpnorm}. Nevertheless, it is easily
seen that $\{\Psi_{x_{j}}\}_{j\in J}$ is a frame for $V$ according
to the above definition if and only if $\{\psi^i_{x_j},\
i=1,2,\dots, t,\ j\in J\}$ is a frame for $V$ according to the
standard definition. The frame bounds, however, may be different.
\end{obs}

%Definition3.4
\begin{defi}\label{def3.4} Let $V$ be a closed subspace of the
Hilbert space $\mathcal{H}$. Let
$\{\Psi_{x_{j}}=(\psi_{x_{j}}^1,\ldots,\psi_{x_{j}}^t)^{T}\}_{j\in
J}$ be a frame for $V$. The frame operator associated with the frame
$\{\Psi_{x_{j}}\}_{j\in J}$ is the operator $S : V\rightarrow V$
defined by $S(f)=\sum_{j\in J}\langle
f,\Psi_{x_{j}}\rangle\Psi_{x_{j}}$, for all $f\in V$.
The (canonical) dual frame $\{\widetilde{\Psi}_{x_{j}}\}_{j\in J}$ of
the frame $\{\Psi_{x_{j}}\}_{j\in J}$ is a sequence of vectors given
by
$\{\widetilde{\Psi}_{x_{j}}=(\widetilde{\psi}_{x_{j}}^{1},\ldots,\widetilde{\psi}_{x_{j}}^{t})^{T}\}_{j\in
J}$, where $\widetilde{\psi}_{x_{j}}^{s}=S^{-1}\psi_{x_{j}}^{s}$,
$1\leq s\leq t$.
\end{defi}

\begin{obs}
It is well know that  a frame operator $S$ is bounded, invertible, self-adjoint, and positive \cite{DS52}.
Hence, the canonical dual frame is well defined. There may exist other dual frames but we will
refrain from defining the notion.
\end{obs}

The next proposition shows that a frame algorithm can be used to reconstruct a function from its samples.

%Proposition4.3
\begin{prop}\label{prop4.3} Let $\Phi\in(\mathit{W}_{0}^{1})^{(r)}$, $\overrightarrow{\mu}\in(\mathcal{M}(\mathbb{R}^d))^{(t)}$, and $X$ be
 a $\overrightarrow{\mu}$-sampling set for $V^2(\Phi)$.
 Then
there exists a sequence of vectors of functions
$\{\Psi_{x_{j}}\}_{j\in J}$, which is a frame for $V^2(\Phi)$
and $\langle
f,\Psi_{x_{j}}\rangle=(f\ast\overrightarrow{\mu})(x_{j})$ for all
$f\in V^2(\Phi)$ and $j\in J$. Moreover, every function $f\in
V^2(\Phi)$ can be recovered from the sequence of its samples
$\{(f\ast\overrightarrow{\mu})(x_{j})\}_{j\in J}$ via
\begin{equation}\label{3.8.5.1}
f(x)=\sum_{j\in J}(f\ast\overrightarrow{\mu})(x_{j})\widetilde{\Psi}_{x_{j}}(x),
\end{equation}
where $\{\widetilde{\Psi}_{x_{j}}\}_{j\in J}$ is the dual frame of
$\{\Psi_{x_{j}}\}_{j\in J}$ and the series (\ref{3.8.5.1})
converges unconditionally in $V^2(\Phi)$.
\end{prop}

The frame $\{\Psi_{x_{j}}\}_{j\in J}$ constructed in the previous proposition
will be called a \emph{$(\overrightarrow{\mu},X)$-sampling frame} for $V^2(\Phi)$.
The main idea of this section is to use the fact that if such a frame is  localized then
%the sampling frame for $V^2(\Phi)$
it is also a Banach frame \cite{kg1} for $V^p(\Phi)$, $p\in[1,\infty)$.

\begin{obs}\label{obs1}
Observe that, in general, the frame operator $S$ is the product of the \emph{analysis operator}
$T:$ $V\to (\lt(J))^{(t)}$, defined by $Tf = \{\langle f,\Psi_{x_{j}}\rangle\}_{j\in J} = \{(\la f,\psi_{x_j}^1\ra,\dots,\la f,\psi_{x_j}^t\ra)\}_{j\in J}$ and its adjoint, that is $S = T^*T$. Since $\Phi$ generates a Riesz basis,
it is immediate that in case of a $(\overrightarrow{\mu},X)$-sampling frame
its analysis operator is isomorphic to the sampling operator $U = U_{(X,\Phi,\overrightarrow{\mu})}$.
\end{obs}

\begin{defi}\label{def3.2} Let $V$ be a closed subspace of the
Hilbert space $\mathcal{H}$. Let
$\{\Psi_{x_{j}}=(\psi_{x_{j}}^1,\ldots,\psi_{x_{j}}^t)^{T}\}_{j\in
J}$ be a frame for $V$, and
$\{G_{k}=(g_{k}^{1},\ldots,g_{k}^{r})^{T}\}_{k\in\mathbb{Z}^d}$ be a
Riesz basis for $V$, i.e., a condition similar to (\ref{2.1.1}) is satisfied.
We say that the frame $\{\Psi_{x_{j}}\}_{j\in J}$ is (polynomially)
$s$-localized with respect to the Riesz basis
$\{G_{k}\}_{k\in\mathbb{Z}^d}$, %with decay $s>0$
%or simply $s$-localized),
if
\begin{equation}\label{3.8.1}
|\langle G_{k},\Psi_{x_{j}}^{T}\rangle|\leq C_{1}(1+|x_{j}-k|)^{-s},
\end{equation}
and
\begin{equation}\label{3.8.2}
|\langle \widetilde{G}_{k},\Psi_{x_{j}}^{T}\rangle|\leq
C_{2}(1+|x_{j}-k|)^{-s},
\end{equation}
for all $j\in J$ and $k\in\mathbb{Z}^d$. Here, the constants $C_{1},
C_{2}>0$ are independent of $j$ and $k$, $|\langle
G_{k},\Psi_{x_{j}}^{T}\rangle|=\sum_{i=1}^{r}\sum_{l=1}^{t}|\langle
g_{k}^{i},\psi_{x_{j}}^{l}\rangle|$,
$\{\widetilde{G}_{k}\}_{k\in\mathbb{Z}^d}$ is the dual Riesz basis
of $\{G_{k}\}_{k\in\mathbb{Z}^d}$, %$\langle
%G_{k},\Psi_{x_{j}}^{T}\rangle$ is the matrix given by
%\begin{displaymath}
%\langle G_{k},\Psi_{x_{j}}^{T}\rangle=\left(\begin{array}{ccc}
%\langle g_{k}^{1},\psi_{x_{j}}^{1}\rangle&
%\ldots&\langle g_{k}^{1},\psi_{x_{j}}^{t}\rangle\\
%\vdots&&\vdots\\
%\langle g_{k}^{r},\psi_{x_{j}}^{1}\rangle&\ldots&\langle
%g_{k}^{r},\psi_{x_{j}}^{t}\rangle
%\end{array}\right).
%\end{displaymath}
%The matrix
and $|\langle \widetilde{G}_{k},\Psi_{x_{j}}^{T}\rangle|$ is
defined similarly to $|\langle {G}_{k},\Psi_{x_{j}}^{T}\rangle|$.
\end{defi}
\begin{obs}\label{obs3.2.1} Let $V$ be a closed subspace of a
Hilbert space $\mathcal{H}$. Assume that
$\{G_{k}=(g_{k}^{1},\ldots,g_{k}^{r})^{T}\}_{k\in\mathbb{Z}^d}$ is a
Riesz basis for $V$. The dual Riesz basis of the Riesz basis
$\{G_{k}\}_{k\in\mathbb{Z}^d}$ is the sequence of vectors
$\{\widetilde{G}_{k}=(\widetilde{g}_{k}^{1},\ldots,\widetilde{g}_{k}^{r})^{T}\}_{k\in\mathbb{Z}^d}$
satisfying $\langle\widetilde{G}_{k}, G_{l}^{T}\rangle=\delta_{kl}I$, %if $k=l$,
%and $\langle\widetilde{G}_{k}, G_{l}^{T}\rangle=O$, otherwise. Here
where $I$ is the $r\times r$ identity matrix, and $\delta_{kl}$ is the Kronecker delta.
Since a Riesz basis $\{G_{k}\}$ is also a frame, $\{\widetilde{G}_{k}\}$ is, in fact, the canonical dual frame for
$\{G_{k}\}$. In this case it is the unique dual frame.
\end{obs}

\begin{defi}\label{sRiesz}
Let $\Phi=(\phi^1,\ldots,\phi^r)^{T}\in(\mathit{W}_{0}^{1})^{(r)}\subset(L^2(\mathbb{R}^d))^{(r)}$ and $s>d$.
We say that $\Phi$ is an $s$-localized Riesz generator for $V^2(\Phi)$, denoted $\Phi\in \mathcal{W}_s$, if
%be given, and assume that $\Phi$ satisfies the following conditions:
\begin{itemize}
\item $\{\Phi_{k}=\Phi(\cdot-k)\}_{k\in\mathbb{Z}^d}$ generates a Riesz
basis for $V^2(\Phi)$, i.e., condition (\ref{2.1.1}) holds for
$p=2$; %\item $\Phi$ is a continuous function on $\mathbb{R}^d$.
\item The components of  $\Phi$ satisfy the decay
condition
\begin{equation}\label{3.8.3}
|\phi^i(x)|\leq C_{0}^{i}(1+|x|)^{-s},
\end{equation}
for all $1\leq i\leq r$ and some
$C_{0}^{i}>0$ independent of $x\in\mathbb{R}^d$.
\end{itemize}
\end{defi}

\begin{obs}\label{obs3.3}
If $\Phi\in\mathcal W_s$, then \eqref{2.1.1} holds for every $p\in[1,\infty]$ as shown in \cite{aagr1}.
\end{obs}

The following is the main result of  subsection 3.2.

%Theorem3.5
\begin{thm}\label{teo3.5} Let $s>d$, $\Phi\in\mathcal{W}_{s}$,
%for some $\nu_{0}>0$,
and $\overrightarrow{\mu}\in(\mathcal{M}_{s}(\mathbb{R}^d))^{(t)}$. %be given.
Assume that $X$ is a $\overrightarrow{\mu}$-sampling set for $V^2(\Phi)$, and
$\{\Psi_{x_{j}}\}_{j\in J}$ is the $(\overrightarrow{\mu},X)$-sampling frame for $V^2(\Phi)$.
 Then
% there exists a frame $\{\Psi_{x_{j}}\}_{j\in J}$ for $V^2(\Phi)$ such that the following holds:
\begin{itemize}
\item $X$ is a $\overrightarrow{\mu}$-sampling set for $V^p(\Phi)$ for all $p\in[1,\infty]$.
\item If  $\{\widetilde{\Psi}_{x_{j}}\}$ is the dual frame for $\{\Psi_{x_{j}}\}_{j\in J}$, then
%frame $\{\Psi_{x_{j}}\}_{j\in J}$
\begin{equation}\label{3.8.6}
f=\sum_{j\in
J}(f\ast\overrightarrow{\mu})(x_{j})\widetilde{\Psi}_{x_{j}},
\mbox{ for all  } f\in V^p(\Phi),
\end{equation}
where the series converges unconditionally in $V^p(\Phi)$, $p\in[1,\infty)$.
\end{itemize}
\end{thm}
%For $1\leq p\leq\infty$, there exist $0<A_{p}\leq
%B_{p}<\infty$ such that
%\begin{equation}\label{3.8.6.1}
%A_{p}\|f\|_{L^p}\leq\|(f\ast\overrightarrow{\mu})(X)\|_{(\ell^{p}(J))^{(t)}}\leq
%B_{p}\|f\|_{L^p}, \mbox{ for all  } f\in V^p(\Phi).
%\end{equation}

Next, we combine Theorem \ref{teo3.5} with the perturbation results of the previous section.
The proofs are immediate.

%Theorem3.5.1
\begin{thm}\label{teo3.5.1} Let $s>d$, $\Phi\in\mathcal{W}_{s}$,
%for some $\nu_{0}>0$,
and $\overrightarrow{\mu}\in(\mathcal{M}_{s}(\mathbb{R}^d))^{(t)}$.
Assume that $X$ is a separated $\overrightarrow{\mu}$-sampling set for $V^2(\Phi)$. %, and
%$\{\Psi_{x_{j}}\}_{j\in J}$ is the $(\overrightarrow{\mu},X)$-sampling frame for $V^2(\Phi)$.
%Assume that $X$ is a sampling set for $V^2(\Phi)$ and
%$\overrightarrow{\mu}$, and it is also a separated set.
Then there exists $\epsilon_{0}>0$ such that for every
$\Theta\in\mathcal{W}_{s}$ satisfying
$\|\Phi-\Theta\|_{(\mathit{W}^{1})^{(r)}}<\epsilon_{0}$, there
exists a $(\overrightarrow{\mu},X)$-sampling frame $\{\Psi_{x_{j}}\}_{j\in J}$ for $V^2(\Theta)$. Moreover,
\begin{itemize}
\item $X$ is a $\overrightarrow{\mu}$-sampling set for $V^p(\Theta)$ for all $p\in[1,\infty]$.
\item If  $\{\widetilde{\Psi}_{x_{j}}\}$ is the dual frame for $\{\Psi_{x_{j}}\}_{j\in J}$, then
\begin{displaymath}
f=\sum_{j\in
J}(f\ast\overrightarrow{\mu})(x_{j})\widetilde{\Psi}_{x_{j}},
\mbox{ for all  } f\in V^p(\Theta),
\end{displaymath}
where the series converges unconditionally in $V^p(\Theta)$, $p\in[1,\infty)$.
%\end{displaymath}
\end{itemize}
\end{thm}

%Theorem3.5.2
\begin{thm}\label{teo3.5.2}
Let $s>d$, $\Phi\in\mathcal{W}_{s}$,
and $\overrightarrow{\mu}\in(\mathcal{M}_{s}(\mathbb{R}^d))^{(t)}$.
Assume that $X$ is a separated $\overrightarrow{\mu}$-sampling set for $V^2(\Phi)$. %, and
Then there exists $\epsilon_{0}>0$ such that for every
%Let $\Phi\in\mathcal{W}_{s+d+\nu_{0}}$
%for some $\nu_{0}>0$, and
%$\overrightarrow{\mu}\in(\mathcal{M}_{0}(\mathbb{R}^d))^{(t)}$ be
%given. Assume that $X$ is a sampling set for $V^2(\Phi)$ and
%$\overrightarrow{\mu}$, and it is also a separated set. Then there
%exists $\epsilon_{0}>0$ such that for every
$\overrightarrow{\alpha}\in(\mathcal{M}_{s}(\mathbb{R}^d))^{(t)}$
satisfying
$\|\overrightarrow{\mu}-\overrightarrow{\alpha}\|_{(\mathcal{M}(\mathbb{R}^d))^{(t)}}<\epsilon_{0}$,
there exists an $(\overrightarrow{\alpha},X)$-sampling frame $\{\Psi_{x_{j}}\}_{j\in J}$ for $V^2(\Phi)$.
Moreover,
\begin{itemize}
\item $X$ is an $\overrightarrow{\alpha}$-sampling set for $V^p(\Phi)$ for all $p\in[1,\infty]$.
\item If  $\{\widetilde{\Psi}_{x_{j}}\}$ is the dual frame for $\{\Psi_{x_{j}}\}_{j\in J}$, then
\begin{displaymath}
f=\sum_{j\in
J}(f\ast\overrightarrow{\mu})(x_{j})\widetilde{\Psi}_{x_{j}},
\mbox{ for all  } f\in V^p(\Phi),
\end{displaymath}
where the series converges unconditionally in $V^p(\Phi)$, $p\in[1,\infty)$.
\end{itemize}
\end{thm}

%Theorem3.5.3
\begin{thm}\label{teo3.5.3}
Let $s>d$, $\Phi\in\mathcal{W}_{s}$,
and $\overrightarrow{\mu}\in(\mathcal{M}_{s}(\mathbb{R}^d))^{(t)}$.
Assume that $X$ is a separated $\overrightarrow{\mu}$-sampling set for $V^2(\Phi)$. %, and
Then there exists $\epsilon_{0}>0$ such that for every
$\Delta = \{\delta_j,\ j\in J\}$ satisfying $\norm{\Delta}_\infty < \epsilon_{0}$
%Let $\Phi\in\mathcal{W}_{s+d+\nu_{0}}$
%for some $\nu_{0}>0$, and
%$\overrightarrow{\mu}\in(\mathcal{M}_{0}(\mathbb{R}^d))^{(t)}$ be
%given. Assume that $X$ is a sampling set for $V^2(\Phi)$ and
%$\overrightarrow{\mu}$, and it is also a separated set. Then there
%exists $\epsilon_{0}>0$ such that for every
%$\overrightarrow{\alpha}\in(\mathcal{M}_{s}(\mathbb{R}^d))^{(t)}$
%satisfying
%$\|\overrightarrow{\mu}-\overrightarrow{\alpha}\|_{(\mathcal{M}(\mathbb{R}^d))^{(t)}}<\epsilon_{0}$,
there exists a $(\overrightarrow{\mu},X+\Delta)$-sampling frame $\{\Psi_{x_{j}}\}_{j\in J}$ for $V^2(\Phi)$.
Moreover,
\begin{itemize}
\item $X+\Delta$ is a $\overrightarrow{\mu}$-sampling set for $V^p(\Phi)$ for all $p\in[1,\infty]$.
\item If  $\{\widetilde{\Psi}_{x_{j}}\}$ is the dual frame for $\{\Psi_{x_{j}}\}_{j\in J}$, then
\begin{displaymath}
f=\sum_{j\in
J}(f\ast\overrightarrow{\mu})(x_{j}+\delta_j)\widetilde{\Psi}_{x_{j}},
\mbox{ for all  } f\in V^p(\Phi),
\end{displaymath}
where the series converges unconditionally in $V^p(\Phi)$, $p\in[1,\infty)$.
\end{itemize}
\end{thm}

%Theorem3.6
\begin{thm}\label{teo3.6}
Let $s>d$, $\Phi\in\mathcal{W}_{s}$,
and $\overrightarrow{\mu}\in(\mathcal{M}_{s}(\mathbb{R}^d))^{(t)}$.
Assume that $X$ is a separated $\overrightarrow{\mu}$-sampling set for $V^2(\Phi)$. %, and
Then there exists $\epsilon_{0}>0$ such that for every
$\Theta\in\mathcal{W}_{s}$ and
$\overrightarrow{\alpha}\in(\mathcal{M}_{s}(\mathbb{R}^d))^{(t)}$
satisfying
$\|\Phi-\Theta\|_{(\mathit{W}^{1})^{(r)}}+\|\overrightarrow{\mu}-\overrightarrow{\alpha}\|_{(\mathcal{M}(\mathbb{R}^d))^{(t)}}<\epsilon_{0}$,
there exists an $(\overrightarrow{\alpha},X)$-sampling frame $\{\Psi_{x_{j}}\}_{j\in J}$ for $V^2(\Theta)$. Moreover,
%such that the following holds:
\begin{itemize}
\item $X$ is an $\overrightarrow{\alpha}$-sampling set for $V^p(\Theta)$ for all $p\in[1,\infty]$.
\item If  $\{\widetilde{\Psi}_{x_{j}}\}$ is the dual frame for $\{\Psi_{x_{j}}\}_{j\in J}$, then
\begin{displaymath}
f=\sum_{j\in
J}(f\ast\overrightarrow{\alpha})(x_{j})\widetilde{\Psi}_{x_{j}},
\mbox{ for all  } f\in V^p(\Theta),
\end{displaymath}
where the series converges unconditionally in $V^p(\Theta)$, $p\in[1,\infty)$.
\end{itemize}
\end{thm}

%Theorem3.6.1
\begin{thm}\label{teo3.6.1}
Let $s>d$, $\Phi\in\mathcal{W}_{s}$,
and $\overrightarrow{\mu}\in(\mathcal{M}_{s}(\mathbb{R}^d))^{(t)}$.
Assume that $X$ is a separated $\overrightarrow{\mu}$-sampling set for $V^2(\Phi)$. %, and
Then there exists $\epsilon_{0}>0$ such that for every
$\Delta = \{\delta_j,\ j\in J\}$,
$\Theta\in\mathcal{W}_{s}$, and
$\overrightarrow{\alpha}\in(\mathcal{M}_{s}(\mathbb{R}^d))^{(t)}$
satisfying
$\norm{\Delta}_\infty+\|\Phi-\Theta\|_{(\mathit{W}^{1})^{(r)}}+\|\overrightarrow{\mu}-\overrightarrow{\alpha}\|_{(\mathcal{M}(\mathbb{R}^d))^{(t)}}<\epsilon_{0}$,
there exists an $(\overrightarrow{\alpha},X+\Delta)$-sampling frame $\{\Psi_{x_{j}}\}_{j\in J}$ for $V^2(\Theta)$. Moreover,
%such that the following holds:
\begin{itemize}
\item $X+\Delta$ is an $\overrightarrow{\alpha}$-sampling set for $V^p(\Theta)$ for all $p\in[1,\infty]$.
\item If  $\{\widetilde{\Psi}_{x_{j}}\}$ is the dual frame for $\{\Psi_{x_{j}}\}_{j\in J}$, then
\begin{displaymath}
f=\sum_{j\in
J}(f\ast\overrightarrow{\alpha})(x_{j}+\delta_j)\widetilde{\Psi}_{x_{j}},
\mbox{ for all  } f\in V^p(\Theta),
\end{displaymath}
where the series converges unconditionally in $V^p(\Theta)$, $p\in[1,\infty)$.
\end{itemize}
\end{thm}

\begin{obs}
The crucial result for the proof of the theorems in this section is
Jaffard's non-commutative extension of the classical Wiener's Tauberian Lemma (see Theorem 5 in \cite{kg1}).
It states that if an invertible matrix has an off-diagonal decay defined by inequalities similar to
\eqref{3.8.1} and \eqref{3.8.2}, then the inverse matrix has the same off-diagonal decay. There exist
other extensions of Wiener's Lemma which deal with different types of off-diagonal decay (see, for example,
\cite{bas,kgml}). Many of those could be used to obtain results similar to, say, Theorem \ref{teo3.6.1}.
\end{obs}
%++++++++++++++++++++++++++++++++++++++++++++++++++++++++++++++++++++++
\subsection{Imperfect reconstruction.}\label{jitter}\

In practice, we know that a perturbation exists because of imperfections of measuring devices, errors, etc. However,
we can only estimate this perturbation and may not even know its nature. Here we show that even if we use a
model $(X,\Phi,\overrightarrow{\mu})$ for reconstructing a signal from a perturbed model $(\widetilde{X},\Theta,\overrightarrow{\alpha})$ (or vice versa), the
reconstruction error depends continuously on the perturbation in the cases studied above.  %The proofs

As before, let $U$ be the sampling operator for a $p$-stable sampling model $(X,\Phi,\overrightarrow{\mu})$ and
$U_{\Delta}$ be the sampling operator for a perturbed model $(\widetilde{X},\Theta,\overrightarrow{\alpha})$,
where $\widetilde{X} = X+\Delta = \{x_j +\delta_j\}_{j\in J}$.
%where $\Delta=\{\delta_{j}\}_{j\in J}\subset\mathbb{R}^d$.
The sampling operator $U_{\Delta}$ can be thought of as a $t\times r$ matrix of
operators given by
\begin{displaymath}
U_{\Delta}
=\left(\begin{array}{ccc} U_{\Delta}^{1,1} &
\ldots&U_{\Delta}^{r,1}\\
\vdots&&\vdots\\
U_{\Delta}^{1,t}&\ldots&U_{\Delta}^{r,t}
\end{array}\right),
\end{displaymath}
where for each $1\leq i\leq r$ and $1\leq l\leq t$ the operator
$U_{\Delta}^{i,l}$ is defined by a bi-infinite matrix with
%$j$, $k$
entries $(U_{\Delta}^{i,l})_{j,k}=(\theta^{i}\ast\alpha^{l})(x_{j}+\delta_{j}-k)$, $j\in J$,
$k\in\mathbb{Z}^d$.

%\begin{obs}\label{obs3.1}
We let $U^{*}$ be an operator defined by the following $r\times t$ matrix of
operators from $(\ell^p(J))^{(t)}$ into $(\ell^p(\mathbb{Z}^d))^{(r)}$:
%given by
\begin{displaymath}
U^{*}=\left(\begin{array}{ccc} \overline{U^{1,1}}&
\ldots&\overline{U^{1,t}}\\
\vdots&&\vdots\\
\overline{U^{r,1}}&\ldots&\overline{U^{r,t}}
\end{array}\right),
\end{displaymath}
where for each $1\leq i\leq r$ and $1\leq l\leq t$, the operator
$\overline{U^{i,l}}$ is defined by a bi-infinite matrix with entries
%$j,k$ entry is
$(\overline{U^{i,l}})_{j,k}=\overline{(\phi^{i}\ast\mu^{l})(x_{j}-k)}$,
where $\overline{z}$ denotes the conjugate of the complex number
$z$. The operator $(U_{\Delta})^{*}$ is defined similarly.
%  the matrix of operators
%associated to $U_{\Delta}$. Also
Notice that this definition implies $\|U^{*}\|_{p,op}=\|U\|_{p,op}$, and $(U^{*})^{*}=U$.
Moreover, if $U$ satisfies (\ref{5.1}), then $U^*$ satisfies
\begin{equation}\label{5.1.1}
\eta_{p}\|D\|_{(\ell^p(J))^{(t)}}\leq\|U^*D\|_{(\ell^p(\mathbb{Z}^d))^{(r)}}\leq\beta_{p}\|D\|_{(\ell^p(J))^{(t)}},
\end{equation}
for all $D\in (\ell^p(J))^{(t)}$.
Observe also that if $p =2$ then $U^*$ is, indeed, the Hilbert adjoint of $U$. Hence,
if the sampling model  $(X,\Phi,\overrightarrow{\mu})$ is $2$-stable, $U^* U$ is isomorphic to the
frame operator $S$ for the sampling frame $\{\Psi_{x_j}\}$, see Remark \ref{obs1}.
Therefore, $U^* U$ is invertible and positive. Moreover, the operator $(U^*U)^{-1}U$
is a left inverse for the sampling operator $U$ and it is isomorphic to the synthesis operator
used for the reconstruction. Hence, the importance of the following result.
%\end{obs}

%Theorem 5.2
\begin{thm}\label{teo5.2} Let $(X,\Phi,\overrightarrow{\mu})$ be a $p$-stable sampling model for some $p\in[1,\infty]$.
Assume that its sampling operator $U$ satisfies (\ref{5.1}) and the operator $U^* U$ is invertible.
Let $\epsilon\in(0,-\beta_p+\sqrt{\beta_p^{2}+\eta_p^{2}})$ %, where
%$0<\eta\leq\beta<\infty$ are constants satisfying (\ref{5.1}) when $p=2$.
%Assume there exists a number $\gamma_{0}>0$ such that
and $(\widetilde{X},\Theta,\overrightarrow{\alpha})$ be a perturbed sampling model such that its sampling operator
$U_\Delta$ satisfies $\|U-U_{\Delta}\|<\epsilon$. %, whenever $\|\Delta\|_{\infty} \leq \gamma_{0}$, and
Define $\nu=\nu(\epsilon)=\eta_p^{-2}\epsilon(\epsilon+2\beta_p)$. Then
$0<\nu<1$, the operator $U_{\Delta}^{*}U_{\Delta}$ is invertible, and
\begin{displaymath}
\|(U^{*}U)^{-1}U^{*}-(U_{\Delta}^{*}U_{\Delta})^{-1}U_{\Delta}^{*}\|<\frac{1}{\eta_p^{2}}\left(\epsilon+\frac{\nu(\beta_p+\epsilon)}{1-\nu}\right).%,\textrm{whenever}\quad\|\Delta\|_{\infty}\leq \gamma_{0}.
\end{displaymath}
\end{thm}

\begin{obs}\label{obs5}
Observe that if $p = 2$ we do not need to require invertibility of $U^* U$. As we mentioned above,
it follows automatically.
%From theorem \ref{teo5.1} it follows that the existence of
%the number $\gamma_{0}>0$ in theorem \ref{teo5.2} is guaranteed.
\end{obs}

\begin{obs}\label{obs6} If in Theorem \ref{teo5.2} we let $r=t=1$, $p = 2$,
and $\mu=\mu^{1}=\delta_{0}$, then we obtain an analog of Theorem 3.3 in
\cite{aacl}.
\end{obs}

Let $(X,\Phi,\overrightarrow{\mu})$ be a $p$-stable sampling model for some $p\in[1,\infty]$.
Assume that its sampling operator $U$ satisfies (\ref{5.1}) and the operator $U^* U$ is invertible.
%Let $\Phi\in(\mathit{W}_{0}^{1})^{(r)}$ and
%$\overrightarrow{\mu}\in(\mathcal{M}(\mathbb{R}^d))^{(t)}$ be given,
%and let $U$ be the sampling operator defined as above when $p=2$.
%If the operator $U^*U$ is invertible,
We define the \emph{reconstruction operator} $R =
R_{(X,\Phi,\overrightarrow{\mu})}: (\lp(J))^{(t)}\rightarrow
V^p(\Phi)$ by
\begin{displaymath}
R D=\sum_{k\in
\mathds{Z}^d}[(U^{*}U)^{-1}U^{*}D]_{k}^{T}\Phi(\cdot-k),
\end{displaymath}
$D=(d^{1},\ldots,d^{t})^{T}$ in $(\lp(J))^{(t)}$.

Then as
an immediate consequence of Theorems \ref{bigshell} and \ref{teo5.2}, we have
the following result.

\bt\label{lastbig}
Let $(X,\Phi,\overrightarrow{\mu})$ be a $p$-stable sampling model for some $p\in[1,\infty]$.
Assume that its sampling operator $U$ is such that $U^* U$ is invertible.
Let $R$ be the reconstruction operator. %Let also $(X+\Delta,\Theta,\overrightarrow{\alpha})$ be a perturbed sampling model.
Then for every $\epsilon > 0$ there exists $\epsilon_0 > 0$ such that for every
$\Delta = \{\delta_j,\ j\in J\}$,
$\Theta\in(\mathit{W}_0^{1})^{(r)}$, and
$\overrightarrow{\alpha}\in(\mathcal{M}(\mathbb{R}^d))^{(t)}$
satisfying
\[\norm{\Delta}_\infty+\|\Phi-\Theta\|_{(\mathit{W}^{1})^{(r)}}+\|\overrightarrow{\mu}-\overrightarrow{\alpha}\|_{(\mathcal{M}(\mathbb{R}^d))^{(t)}}<\epsilon_{0},\]
we have
\[
\norm{R((g\ast\overrightarrow{\alpha})(X+\Delta))-f}_{L^p} < \epsilon,\ f = \sum_{k\in\GG} C_k^T\Phi_k,\ g = \sum_{k\in\GG} C_k^T\Theta_k,
\]
for all $C\in(\lp(\GG))^{(r)}$.
\et

Theorem \ref{lastbig} tells us that the reconstruction error is, indeed, controlled in a continuous fashion by each and all of the perturbation
errors studied in this paper.

Our final result is a combination of the above theorem with the results of section 3.2.

\bt\label{lastcombine}
Let $(X,\Phi,\overrightarrow{\mu})$ be a $2$-stable sampling model such that $\Phi\in\mathcal{W}_{s}$ and
$\overrightarrow{\mu}\in(\mathcal{M}_{s}(\mathbb{R}^d))^{(t)}$.
%Assume that its sampling operator $U$ is such that $U^* U$ is invertible.
Let $R$ be the reconstruction operator for $(X,\Phi,\overrightarrow{\mu})$. %Let also $(X+\Delta,\Theta,\overrightarrow{\alpha})$ be a perturbed sampling model.
Then for every $\epsilon > 0$ there exists $\epsilon_0 > 0$ such that for every
$\Delta = \{\delta_j,\ j\in J\}$,
$\Theta\in\mathcal{W}_{s}$, and
$\overrightarrow{\alpha}\in(\mathcal{M}_{s}(\mathbb{R}^d))^{(t)}$
satisfying
\[\norm{\Delta}_\infty+\|\Phi-\Theta\|_{(\mathit{W}^{1})^{(r)}}+\|\overrightarrow{\mu}-\overrightarrow{\alpha}\|_{(\mathcal{M}(\mathbb{R}^d))^{(t)}}<\epsilon_{0},\]
we have
\[
\norm{R((g\ast\overrightarrow{\alpha})(X+\Delta))-f}_{L^p} < \epsilon,\ f = \sum_{k\in\GG} C_k^T\Phi_k,\ g = \sum_{k\in\GG} C_k^T\Theta_k,
\]
for all $p\in[1,\infty]$ and all $C\in(\lp(\GG))^{(r)}$.
\et

The proofs in the following section show implicitly how numerical estimates for $\epsilon_0$ in Theorems \ref{lastbig} and \ref{lastcombine} may be obtained.

\section{Proofs}

\subsection{Auxiliary results.}\

We begin with technical results that are needed for the main proofs.

%Lemma4.1
\begin{lem}\label{lem4.1}
Let $\phi\in\mathit{W}_{0}^{1}$ , and
$\mu\in\mathcal{M}(\mathbb{R}^d)$. Then:
\begin{equation}\label{4.1.1}
\phi\ast\mu\in\mathit{W}_{0}^{1},\mbox { and }
\end{equation}
\begin{equation}\label{4.1.2}
\|\phi\ast\mu\|_{\mathit{W}^{1}}\leq
2^{d}\|\phi\|_{\mathit{W}^{1}}\|\mu\| .
\end{equation}
\end{lem}
\begin{proof} Note that if $\mu=0$, %we are done, because in this case
%$\phi\ast\mu\equiv0$, and clearly it is a continuous function in
%$\mathbb{R}^{d}$ and belongs to $\mathit{W}^1$, and obviously
%inequality (\ref{4.1.2}) takes place.
the proof is immediate.
Assume now $\mu\neq0$, i.e.
$\|\mu\|>0$. Let $\epsilon>0$ be given. Since
$\phi\in\mathit{W}_{0}^{1}$, then $\phi$ is uniformly continuous
in $\mathbb{R}^d$. Therefore, there exists
$\delta=\delta(\epsilon)>0$ such that
\begin{equation}\label{4.1.3}
|\phi(w)-\phi(w_{1})|<\frac{\epsilon}{\|\mu\|},\quad\textrm{whenever}\quad
\|w-w_{1}\|<\delta .
\end{equation}
Let $z_{0}\in\mathbb{R}^{d}$ be given, and let $z\in\mathbb{R}^d$
be such that $\|z-z_{0}\|<\delta $. Then we have
\begin{eqnarray}
\left|(\phi\ast\mu)(z)-(\phi\ast\mu)(z_{0})\right|&=&\left|\int_{\mathbb{R}^d}\phi(z-y)d\mu(y)-\int_{\mathbb{R}^d}\phi(z_{0}-y)d\mu(y)\right|\nonumber\\
&=&\left|\int_{\mathbb{R}^d}(\phi(z-y)-\phi(z_{0}-y))d\mu(y)\right| \nonumber\\
&\leq& \int_{\mathbb{R}^d}|\phi(z-y)-\phi(z_{0}-y)|d|\mu|(y).
\nonumber
\end{eqnarray}
Since $\|(z-y)-(z_{0}-y)\|=\|z-z_{0}\|<\delta$, for all
$y\in\mathbb{R}^d$, then  it follows from (\ref{4.1.3}) that
$\int_{\mathbb{R}^d}|(\phi(z-y)-\phi(z_{0}-y))|\,
d|\mu|(y)<\int_{\mathbb{R}^d}\frac{\epsilon}{\|\mu\|} d|\mu|(y)
=\epsilon$. Since $z_{0}$ and $\epsilon>0$ are arbitrary, we
obtain the continuity of $\phi\ast\mu$ in $\mathbb{R}^d$.

Let us show (\ref{4.1.2}). Let $\phi\in\mathit{W}^1$ and
$\mu\in\mathcal{M}(\mathbb{R}^d)$ be given. Then %we have
%\newpage
\begin{eqnarray}
&\|\phi\ast\mu\|_{\mathit{W}^1} = \sum\limits_{k\in\mathbb{Z}^d}\underset{{x\in[0,1]^{d}}}{\operatorname{esssup}}\,\left|\int\limits_{\mathbb{R}^d}\phi(x+k-y)
\,d\mu(y)\right| \leq \nonumber\\
%\leq&\sum_{k\in\mathbb{Z}^d}\underset{{x\in[0,1]^{d}}}{\operatorname{esssup}}\,\int_{\mathbb{R}^d}|\phi(x+k-y)|d|\mu|(y)\nonumber \\
%&\leq&\sum_{k\in\mathbb{Z}^d}\operatorname{esssup}_{x\in[0,1]^{d}}\,\int_{\mathbb{R}^d}\operatorname{esssup}_{x\in[0,1]^{d}}\,|\phi(x+k-y)|d|\mu|(y)\nonumber\\
&\sum\limits_{k\in\mathbb{Z}^d}\int\limits_{\mathbb{R}^d}\underset{{x\in[0,1]^{d}}}{\operatorname{esssup}}\,|\phi(x+k-y)|\,d|\mu|(y)\leq\nonumber\\
&\int\limits_{\mathbb{R}^d}\left(\sum\limits_{k\in\mathbb{Z}^d}\underset{{x\in[0,1]^{d}}}{\operatorname{esssup}}\,|\phi(x+k-y)|\right)d|\mu|(y)%\nonumber\\
=\int\limits_{\mathbb{R}^d}\|\phi(\cdot-y)\|_{\mathit{W}^1}\,d|\mu|(y).\nonumber
\end{eqnarray}
Since
%\begin{displaymath}
$\|\phi(\cdot-y)\|_{\mathit{W}^1}\leq2^{d}\|\phi\|_{\mathit{W}^1}$,  for all  $y\in \mathbb{R}^d$,
%\end{displaymath}
%then it follows
we get
%\begin{eqnarray}
\[\int_{\mathbb{R}^d}\|\phi(\cdot-y)\|_{\mathit{W}^1}d|\mu|(y)\leq\int_{\mathbb{R}^d}2^{d}\|\phi\|_{\mathit{W}^1}d|\mu|(y)%\nonumber\\
=2^{d}\|\phi\|_{\mathit{W}^1}\|\mu\|.
\]%\end{eqnarray}
Therefore, we get (\ref{4.1.2}).
\end{proof}

The next proposition collects basic facts about Wiener amalgam spaces, shift invariant spaces
$V^p(\Phi)$, and separated sets in $\RR^d$.

%Proposition4.2
\begin{prop}\label{prop4.2}
Let $\Phi\in(\mathit{W}_{0}^{1})^{(r)}$,
$\overrightarrow{\mu}\in(\mathcal{M}(\mathbb{R}^d))^{(t)}$, %and
$f=\sum\limits_{k\in\mathbb{Z}^d}C_{k}^{T}\Phi_{k}$, where
$C\in(\ell^{p}(\mathds{Z}^d))^{(r)}$, and $\Phi_{k}=\Phi(\cdot-k)$, for
all $k\in\mathbb{Z}^d$. Let also $X=\{x_{j},j\in J\}$ be a separated
set in $\mathbb{R}^d$ with a separation constant $\delta>0$. Then
\begin{equation}\label{4.2.1}
\Phi\ast\overrightarrow{\mu}\in(\mathit{W}_{0}^{1})^{(r\times t)};
\end{equation}
\begin{equation}\label{4.2.2}
\|\Phi\ast\overrightarrow{\mu}\|_{(\mathit{W}^{1})^{(r\times
t)}}\leq2^{d}\|\Phi\|_{(\mathit{W}^{1})^{(r)}}\|\overrightarrow{\mu}\|_{(\mathcal{M}(\mathbb{R}^d))^{(t)}};
\end{equation}
\begin{equation}\label{4.2.3}
V^{p}(\Phi)\subset\mathit{W}_{0}^{p},\mbox{ for all }
1\leq p \leq\infty;
\end{equation}
\begin{equation}\label{4.2.4}
\|f\|_{\mathit{W}^{p}}\leq\|C\|_{(\ell^{p}(\mathbb{Z}^d))^{(r)}}\|\Phi\|_{(\mathit{W}^{1})^{(r)}};
\end{equation}
\begin{equation}\label{4.2.5}
\|f(X)\|_{\ell^{p}(J)}\leq\mathcal{N}\|f\|_{\mathit{W}^{p}},\mbox{ where }
\mathcal{N}=\mathcal{N}(\delta,p,d)=(\frac{\sqrt{d}}{\delta}+1)^{d/p}.
\end{equation}
\end{prop}

\begin{proof} First, Lemma \ref{lem4.1} immediately implies (\ref{4.2.1}).

Next, to prove (\ref{4.2.2}) consider
$\Phi=(\phi^{1},\ldots,\phi^{r})^{T}\in(\mathit{W}_{0}^{1})^{r}$
and $\overrightarrow{\mu}=(\mu^{1},\ldots,\mu^{t})\in
(\mathcal{M}(\mathbb{R}^d))^{(t)}$. Then
%\begin{displaymath}
%\|\Phi\ast\overrightarrow{\mu}\|_{(\mathit{W}^{1})^{(r\times
%t)}}=\sum_{j=1}^{t}\sum_{i=1}^{r}
%\|\phi^{i}\ast\mu^{j}\|_{\mathit{W}^{1}}.
%\end{displaymath}
%The proof of (\ref{4.2.1}) and (\ref{4.2.2}) follows immediately from
%Lemma \ref{lem4.1}, because since
%$\Phi=(\phi^{1},\ldots,\phi^{r})^{T}\in(\mathit{W}_{0}^{1})^{r}$ if
%and only if $\phi^{i}\in\mathit{W}_{0}^{1}$, for all $1 \leq i \leq
%r$, and $\overrightarrow{\mu}=(\mu^{1},\ldots,\mu^{t})\in
%(\mathcal{M}(\mathbb{R}^d)^{(t)}$ if and only if
%$\mu^{j}\in\mathcal{M}(\mathbb{R}^d)$, for all $1 \leq j \leq t$,
%then by (\ref{4.1.1}) in lemma \ref{lem4.1} we have
%$\phi^{i}\ast\mu^{j}$ are continuous functions in $\mathbb{R}^d$,
%for all $1 \leq i \leq r$ and $1 \leq j \leq t$. Thus, the matrix
%$\Phi\ast\overrightarrow{\mu}$ given by
%\begin{displaymath}
%\Phi\ast\overrightarrow{\mu}=\left(\begin{array}{ccc}
%\phi^{1}\ast\mu^{1}&
%\ldots&\phi^{1}\ast\mu^{t}\\
%\vdots&&\vdots\\
%\phi^{r}\ast\mu^{1}&\ldots&\phi^{r}\ast\mu^{t}
%\end{array}\right)
%\end{displaymath}
%is a matrix of continuous functions in $\mathbb{R}^d$. On the
%other hand,
using (\ref{4.1.2}), we obtain %in lemma \ref{lem4.1}, we obtain:
\begin{eqnarray}
\|\Phi\ast\overrightarrow{\mu}\|_{(\mathit{W}^{1})^{(r\times
t)}}&=&\sum_{j=1}^{t}\sum_{i=1}^{r}
\|\phi^{i}\ast\mu^{j}\|_{\mathit{W}^{1}}\leq\nonumber\\
%&\leq&
\sum_{j=1}^{t}\sum_{i=1}^{r}2^{d}\|\phi^{i}\|_{\mathit{W}^{1}}\|\mu^{j}\|%\nonumber\\
&=&2^{d}\|\Phi\|_{(\mathit{W}^{1})^{(r)}}\|\overrightarrow{\mu}\|_{(\mathcal{M}(\mathbb{R}^d))^{(t)}}.\nonumber
\end{eqnarray}
%Note that (\ref{4.2.3}) follows from (\ref{4.2.4}).

Next, we prove
(\ref{4.2.4}). Consider $1 \leq p <\infty$ and
$f=\sum\limits_{k\in\mathbb{Z}^d}C_{k}^{T}\Phi_{k}$. For each $1
\leq s \leq r$ let
$a^{s}(l)=\underset{{x\in[0,1]^{d}}}{\operatorname{esssup}}\,|\phi^{s}(x+l)|$,
for all $l\in\mathbb{Z}^d$. Then
$\|a^{s}\|_{\ell^{1}(\mathbb{Z}^d)}=\|\phi^{s}\|_{\mathit{W}^{1}}$.
Consequently,
$\|a\|_{(\ell^{1}(\mathbb{Z}^d))^{(r)}}=\|\Phi\|_{(\mathit{W}^{1})^{(r)}}$,
where $a=(a^{1},\ldots,a^{r})^{T}$, and
$\Phi=(\phi^{1},\ldots,\phi^{r})^{T}$. Hence,
\[
\underset{{x\in[0,1]^{d}}}{\operatorname{esssup}}\,|f(x+l)|\leq\sum_{s=1}^{r}\sum_{k\in\mathbb{Z}^d}|c^{s}(k)|
\underset{{x\in[0,1]^{d}}}{\operatorname{esssup}}\,|\phi^{s}(x+l-k)|=\sum_{s=1}^{r}(a^{s}\ast|c^{s}|)(l).\nonumber
\]
%\begin{eqnarray}
%\underset{{x\in[0,1]^{d}}}{\operatorname{esssup}}\,|f(x+l)|&\leq&\sum_{s=1}^{r}\sum_{k\in\mathbb{Z}^d}|c^{s}(k)|
%\underset{{x\in[0,1]^{d}}}{\operatorname{esssup}}\,|\phi^{s}(x+l-k)|\nonumber\\
%&=&\sum_{s=1}^{r}(a^{s}\ast|c^{s}|)(l).\nonumber
%\end{eqnarray}
By using Young and triangular inequalities, we have
\begin{displaymath}
\|f\|_{\mathit{W}^{p}}\leq
\sum_{s=1}^{r}\|a^{s}\ast|c^{s}|\|_{\ell^{p}}\leq
\sum_{s=1}^{r}\|a^{s}\|_{\ell^{1}}\|c^{s}\|_{\ell^{p}}.
\end{displaymath}
Consequently,
$\|f\|_{\mathit{W}^{p}}\leq\|C\|_{(\ell^{p}(\mathbb{Z}^d))^{r}}\|\Phi\|_{(\mathit{W}^{1})^{(r)}}$.

Next, let us show (\ref{4.2.3}). Let $f\in V^p(\Phi)$ be given. Then
$f=\sum_{k\in\mathbb{Z}^d}C_{k}^{T}\Phi_{k}$, for some
$C\in(\ell^p(\mathbb{Z}^d))^{(r)}$. Since (\ref{4.2.4}) implies
$f\in\mathit{W}^p$, it remains to show the continuity of $f$. Let us
first consider  the case $1\leq p<\infty$. We observe that
$\mathit{W}^p\subset\mathit{W}^{\infty}=L^{\infty}(\mathbb{R}^d)$
(see Theorem 2.1 in \cite{aagr2}), and, hence,
\begin{equation}\label{4.2.3.1}
\|f\|_{L^{\infty}(\mathbb{R}^d)}\leq d_{1}\|f\|_{\mathit{W}^p},
\end{equation}
for some $d_{1}>0$ independent of $f$. Let $f_{n}=\sum_{|k|\leq
n}C_{k}^{T}\Phi_{k}$ be a partial sum of $f$. Since
$\Phi\in(\mathit{W}_{0}^{1})^{(r)}$, then
$\{f_{n}\}_{n\in\mathbb{N}}$ is a sequence of continuous
functions, and from (\ref{4.2.4}) and (\ref{4.2.3.1}) we obtain
\begin{displaymath}
\|f-f_{n}\|_{L^{\infty}(\mathbb{R}^d)}\leq
d_{1}\|\Phi\|_{(\mathit{W}^1)^{(r)}}\left(\sum_{i=1}^{r}\left(\sum_{|k|>n}|c_{k}^{i}|^{p}\right)^{1/p}\right).
\end{displaymath}
Therefore, the sequence of continuous functions
$\{f_{n}\}_{n\in\mathbb{N}}$ converges uniformly to the function
$f$. Thus, $f$ is a continuous function as well. To treat the case
$p=\infty$, we choose a sequence $\{\Phi_{n}\}_{n\geq 1}$ of
continuous functions with compact support (see Theorem 3.1 in
\cite{aagr2} for details) such that
$\|\Phi_{n}-\Phi\|_{(\mathit{W}^1)^{(r)}}\rightarrow0$ as
$n\rightarrow\infty$. Set
$f_{n}(x)=\sum_{k\in\mathbb{Z}^d}C_{k}^{T}\Phi_{n}(x-k)$. Since the
sum is locally finite, then each $f_{n}$ is continuous. By using
(\ref{4.2.4}) once again, we estimate
\begin{displaymath}
\|f_{n}-f\|_{L^{\infty}(\mathbb{R}^d)}\leq
d_{1}\|C\|_{(\ell^{\infty}(\mathbb{Z}^d))^{(r)}}\|\Phi_{n}-\Phi\|_{(\mathit{W}^1)^{(r)}}.
\end{displaymath}
It follows that the sequence of continuous functions
$\{f_{n}\}_{n\geq 1}$ converges uniformly to $f$. Hence, $f$ is a
continuous function as well.

Finally, let us prove (\ref{4.2.5}). Since
$X=\{x_{j},j\in J\}\subset\mathbb{R^{d}}$ is separated with
a separation constant $\delta>0$, then $\inf_{j\neq
k}|x_{j}-x_{k}|\geq\delta$. Consequently, there exist at most
$([\frac{\sqrt{d}}{\delta}]+1)^{d}$ sampling points in every
$d$-dimensional hypercube $[0,1]^{d}+l$, %for all
$l\in\mathbb{Z}^d$. Therefore,
\[
\sum_{j : x_{j}\in[0,1]^{d}+l}|f(x_{j})|^{p}\leq
(\delta^{-1}\sqrt{d}+1)^{d}\underset{{x\in[0,1]^{d}}}{\operatorname{esssup}}\,|f(x)|^{p},%\nonumber\\
\]%\end{eqnarray}
and, hence,
%\begin{displaymath}
$\|f(X)\|_{\ell^{p}(J)}\leq\mathcal{N}\|f\|_{\mathit{W}^{p}}$, for all
$f\in W^{p}$, where $\mathcal{N}=(\delta^{-1}\sqrt{d}+1)^{d/p}$.
%\end{displaymath}
\end{proof}

Using (\ref{4.2.1}) and (\ref{4.2.2}), %in proposition \ref{prop4.2},
we obtain the following result.

%COROLLARY 4.1(TO PROPOSITION 4.2)

\begin{coro}\label{coro4.1} Let $\Lambda :(\mathit{W}_{0}^{1})^{(r)}\times(\mathcal{M}(\mathbb{R}^d))^{(t)}\longrightarrow(\mathit{W}_{0}^{1})^{(r\times
t)}$ be defined by
$\Lambda(\Phi,\overrightarrow{\mu})=\Phi\ast\overrightarrow{\mu}$.
Then $\Lambda$ is a bounded bilinear form, and
$\|\Lambda\|\leq 2^{d}$, where %. Here the norm of $\Lambda$ is given by
\begin{displaymath}
\|\Lambda\|=\sup\{\|\Lambda(\Phi,\overrightarrow{\mu})\|_{(\mathit{W}_{0}^{1})^{(r\times
t)}}: \|\Phi\|_{(\mathit{W}_{0}^{1})^{(r)}}\leq 1,
\|\overrightarrow{\mu}\|_{(\mathcal{M}(\mathbb{R}^d))^{(t)}}\leq
1\}.
\end{displaymath}
\end{coro}
The following lemma, proved, for example, in \cite{aaik}, states that a small perturbation
of a Riesz basic sequence remains a Riesz basic sequence.
%Lemma4.3
\begin{lem}\label{lem4.3}
Let $\Phi\in(\mathit{W}^{1})^{(r)}$ satisfy (\ref{2.1.1}). Then
there exists $\epsilon_{0}>0$ such that every
$\Theta\in(\mathit{W}^{1})^{(r)}$ satisfying
$\|\Phi-\Theta\|_{(\mathit{W}^{1})^{(r)}}\leq
\epsilon<\epsilon_{0},$ also satisfies (\ref{2.1.1}), for some
$0<m_{p}^{'}\leq M_{p}^{'}<\infty$ and
%\end{lem}
%From the proof of lemma \ref{lem4.3} given in \cite{aaik} we
%obtain the following estimates for the constants $m_{p}^{'}$ and
%$M_{p}^{'}$, respectively:
\begin{equation}\label{4.3.3}
m_{p}^{'}\geq m_{p}-\epsilon \qquad \textrm{and}\qquad
M_{p}^{'}\leq \|\Phi\|_{(\mathit{W}^{1})^{(r)}}+\epsilon.
\end{equation}
\end{lem}

\subsection{Proofs for Section 3.1.}\

Now we are ready to prove the first of our main results.

\medskip
%PROOF THEOREM 3.1
{\bf Proof of Theorem \ref{teo3.1}.}
\begin{proof}
Assume that $\overrightarrow{\mu}\in(\mathcal{M}(\mathbb{R}^d))^{(t)}$,
$\Phi\in(\mathit{W}_{0}^{1})^{(r)}$ satisfies (\ref{2.1.1}), and
$X=\{x_{j},j\in J\}\subset\mathbb{R}^d$ satisfies \eqref{1.2}.
 We want to find
$\epsilon_{0}>0$ such that whenever
$\|\Phi-\Theta\|_{(\mathit{W}^{1})^{(r)}}\leq\epsilon<\epsilon_{0}$,
then (\ref{3.1.2}) takes place for some $0<A_{p}^{'} \leq
B_{p}^{'}<\infty.$ Assume $0<\epsilon<m_{p}$. Then, by Lemma
\ref{lem4.3},  $\Theta\in(\mathit{W}^{1})^{(r)}$ satisfies
(\ref{2.1.1}) and we can use representations
$g=\sum_{k\in\mathbb{Z}^d}C_{k}^{T}\Theta_{k}$ and
$f=\sum_{k\in\mathbb{Z}^d}C_{k}^{T}\Phi_{k}$,
$C\in(\ell^{p}(\mathbb{Z}^d))^{(r)}$. Consequently, we have
\begin{eqnarray}
\frac{1}{M_{p}^{'}}\|g\|_{L^{p}}&\leq&\|C\|_{(\ell^{p}(\mathbb{Z}^d))^{(r)}}\leq\frac{1}{m_{p}}\|\sum_{k\in\mathbb{Z}^d}C_{k}^{T}\Phi_{k}\|_{L^{p}}=\frac{1}{m_{p}}\|f\|_{L^{p}}\nonumber\\
&\leq&\frac{A_{p}^{-1}}{m_{p}}\|(f\ast\overrightarrow{\mu})(X)\|_{(\ell^{p}(J))^{(t)}}\nonumber\\
&=&\frac{A_{p}^{-1}}{m_{p}}\norm{\left(\left(\sum_{k\in\mathbb{Z}^d}C_{k}^{T}\Phi_{k}\right)\ast\overrightarrow{\mu}\right)(X)}_{(\ell^{p}(J))^{(t)}}\nonumber\\
&=&\frac{A_{p}^{-1}}{m_{p}}\sum_{l=1}^{t}\norm{\left(\left(\sum_{k\in\mathbb{Z}^d}C_{k}^{T}\Phi_{k}\right)\ast\mu^{l}\right)(X)}_{\ell^{p}(J)}\nonumber\\
&\leq&\frac{A_{p}^{-1}}{m_{p}}\sum_{l=1}^{t}\norm{\left(\sum_{k\in\mathbb{Z}^d}C_{k}^{T}\Xi_{k}^{l}\right)(X)}_{\ell^{p}(J)}\nonumber\\
&+&\frac{A_{p}^{-1}}{m_{p}}\norm{\left(g\ast\overrightarrow{\mu}\right)(X)}_{(\ell^{p}(J))^{(t)}},\nonumber
\end{eqnarray}
where
\begin{equation}\label{3.1.3}
\Xi_{k}^{l}:=((\phi_{k}^{1}-\theta_{k}^{1})\ast\mu^{l},\ldots,(\phi_{k}^{r}-\theta_{k}^{r})\ast\mu^{l}),\quad
l=1,\ldots,t.
\end{equation}
Since $\Phi$ and $\Theta$ are elements of
$(\mathit{W}_{0}^{1})^{(r)}$ and
$\overrightarrow{\mu}\in(\mathcal{M}(\mathbb{R}^d))^{(t)}$, then
by %proposition \ref{prop4.2}
(\ref{4.2.1}), we have
$\Xi^{l}=(\Phi-\Theta)\ast\mu^{l}\in(\mathit{W}_{0}^{1})^{(r)}$,
for $l=1,\ldots,t$. Hence, using %proposition \ref{prop4.2}
(\ref{4.2.2}), (\ref{4.2.3}) and condition
(\ref{2.1.1}) for $g=\sum_{k\in\mathbb{Z}^d}C_{k}^{T}\Theta_{k}$,
we have
\begin{eqnarray}
&\sum_{l=1}^{t}\norm{(\sum_{k\in\mathbb{Z}^d}C_{k}^{T}\Xi_{k}^{l})(X)}_{\ell^{p}(J)} \leq \nonumber\\
&2^{d}\mathcal{N}\|C\|_{(\ell^{p}(\mathbb{Z}^d))^{(r)}}\|\Phi-\Theta\|_{(\mathit{W}^{1})^{(r)}}\|\overrightarrow{\mu}\|_{(\mathcal{M}(\mathbb{R}^d))^{(t)}}\leq \nonumber\\
&\frac{2^{d}\mathcal{N}\|\Phi-\Theta\|_{(\mathit{W}^{1})^{(r)}}\|\overrightarrow{\mu}\|_{(\mathcal{M}(\mathbb{R}^d))^{(t)}}}{m_{p}^{'}}\|g\|_{L^{p}}.\nonumber
\end{eqnarray}
Therefore,
\begin{eqnarray}
\frac{1}{M_{p}^{'}}\|g\|_{L^{p}}&\leq&\frac{A_{p}^{-1}2^{d}\mathcal{N}\|\Phi-\Theta\|_{(\mathcal{W}^{1})^{(r)}}\|\overrightarrow{\mu}\|_{(\mathcal{M}(\mathbb{R}^d))^{(t)}}}{m_{p}m_{p}^{'}}\|g\|_{L^{p}}+\nonumber\\
&+&\frac{A_{p}^{-1}}{m_{p}}\|(g\ast\overrightarrow{\mu})(X)\|_{(\ell^{p}(J))^{(t)}}.\nonumber
\end{eqnarray}
Hence,
\begin{equation}\label{3.1.4}
\bs
\left( \frac{A_{p}m_{p}}{M_{p}^{'}} \right. & \left. -  \frac{2^{d}\mathcal{N}\|\Phi-\Theta\|_{(\mathit{W}^{1})^{(r)}}
\|\overrightarrow{\mu}\|_{(\mathcal{M}(\mathbb{R}^d))^{(t)}}}{m_{p}^{'}} \right) \|g\|_{L^{p}} \\
& \leq  \|(g\ast\overrightarrow{\mu})(X)\|_{(\ell^{p}(J))^{(t)}}.
\end{split}
\end{equation}
On the other hand, since  $\Theta\in(\mathit{W}_{0}^{1})^{(r)}$
and $\overrightarrow{\mu}\in(\mathcal{M}(\mathbb{R}^d))^{(t)}$, it
follows from %proposition \ref{prop4.2}
(\ref{4.2.1}) that
$(\theta^{1}\ast\mu^{l},\ldots,\theta^{r}\ast\mu^{l})\in(\mathit{W}_{0}^{1})^{(r)}$,
$l=1,\ldots,t$. Therefore, %proposition \ref{prop4.2}
(\ref{4.2.4}), (\ref{4.2.5}) and the first of the estimates in (\ref{4.3.3}) imply that
\begin{eqnarray}
\|(g\ast\overrightarrow{\mu})(X)\|_{(\ell^{p}(J))^{(t)}}&=&\|((\sum_{k\in\mathbb{Z}^d}C_{k}^{T}\Theta_{k})\ast\overrightarrow{\mu})(X)\|_{(\ell^{p}(J))^{(t)}}\nonumber\\
&\leq&\mathcal{N}\|((\sum_{k\in\mathbb{Z}^d}C_{k}^{T}\Theta_{k})\ast\overrightarrow{\mu})\|_{(\mathit{W}^{p})^{(r)}}\nonumber\\
&\leq&2^{d}\mathcal{N}\|\overrightarrow{\mu}\|_{(\mathcal{M}(\mathbb{R}^d))^{(t)}}\|\sum_{k\in\mathbb{Z}^d}C_{k}^{T}\Theta_{k}\|_{(\mathit{W}^{p})^{(r)}}\nonumber\\
&\leq&2^{d}\mathcal{N}\|\overrightarrow{\mu}\|_{(\mathcal{M}(\mathbb{R}^d))^{(t)}}\|C\|_{(\ell^{p}(\mathbb{Z}^d))^{(r)}}\|\Theta\|_{(\mathit{W}^{1})^{(r)}}\nonumber\\
&\leq&\frac{2^{d}\mathcal{N}\|\overrightarrow{\mu}\|_{(\mathcal{M}(\mathbb{R}^d))^{(t)}}}{m_{p}^{'}}(\|\Phi\|_{(\mathit{W}^{1})^{(r)}}+\epsilon)\|g\|_{L^{p}}\nonumber\\
&\leq&\frac{2^{d}\mathcal{N}\|\overrightarrow{\mu}\|_{(\mathcal{M}(\mathbb{R}^d))^{(t)}}(\|\Phi\|_{(\mathit{W}^{1})^{(r)}}+\epsilon)}{m_{p}-\epsilon}\|g\|_{L^{p}}.\nonumber
\end{eqnarray}
Hence,
\begin{equation}\label{3.1.5}
\|(g\ast\overrightarrow{\mu})(X)\|_{(\ell^{p}(J))^{(t)}}\leq\left(\frac{2^{d}\mathcal{N}\|\overrightarrow{\mu}\|_{(\mathcal{M}(\mathbb{R}^d))^{(t)}}(\|\Phi\|_{(\mathit{W}^{1})^{(r)}}+\epsilon)}{m_{p}-\epsilon}\right)\|g\|_{L^{p}}.
\end{equation}
Using the estimates (\ref{4.3.3}) and the left hand side of the
inequality (\ref{3.1.4}), we can obtain an explicit upper bound $\epsilon_0$
for $\epsilon$ from
\begin{displaymath}
\frac{A_{p}m_{p}}{\|\Phi\|_{(\mathit{W}^{1})^{(r)}}+\epsilon}-\frac{2^{d}\mathcal{N}\|\overrightarrow{\mu}\|_{(\mathcal{M}(\mathbb{R}^d))^{(t)}}}{m_{p}-\epsilon}\epsilon=0.
\end{displaymath}
This is equivalent to the quadratic equation
\[
\epsilon^{2} +C_{p}\epsilon
-\frac{A_{p}m_{p}^{2}}{2^{d}\mathcal{N}\|\overrightarrow{\mu}\|_{(\mathcal{M}(\mathbb{R}^d))^{(t)}}}=0,\]
where
\[ C_{p}=\|\Phi\|_{(\mathit{W}^{1})^{(r)}}
+\frac{A_{p}m_{p}}{2^{d}\mathcal{N}\|\overrightarrow{\mu}\|_{(\mathcal{M}(\mathbb{R}^d))^{(t)}}}.\]
%\end{equation}
Let $\epsilon_{0}$ be the positive solution of the previous equation, i.e.,
\begin{displaymath}
\epsilon_{0}=\frac{1}{2}\left(\sqrt{C_{p}^{2}+\frac{4A_{p}m_{p}^{2}}{2^{d}\mathcal{N}\|\overrightarrow{\mu}\|_{(\mathcal{M}(\mathbb{R}^d))^{(t)}}}}-C_{p}\right).
\end{displaymath}
Then, for $0<\epsilon<\epsilon_{0}<m_p$, we  use (\ref{3.1.4}),
(\ref{3.1.5}), and (\ref{4.3.3}) to obtain
\begin{displaymath}
A_{p}^{'}=\frac{A_{p}m_{p}}{\|\Phi\|_{(\mathit{W}^{1})^{(r)}}+\epsilon}-\frac{2^{d}\mathcal{N}\|\overrightarrow{\mu}\|_{(\mathcal{M}(\mathbb{R}^d))^{(t)}}}{m_{p}-\epsilon}\epsilon,%\quad\textrm{and}
\end{displaymath}
\begin{displaymath}
B_{p}^{'}=\frac{2^{d}\mathcal{N}\|\overrightarrow{\mu}\|_{(\mathcal{M}(\mathbb{R}^d))^{(t)}}(\|\Phi\|_{(\mathit{W}^{1})^{(r)}}+\epsilon)}{m_{p}-\epsilon},
\end{displaymath}
and the proof is complete.
\end{proof}

%PROOF COROLLARY 3.4
%{\bf Proof of corollary \ref{coro3.4}.}
%\begin{proof} Corollary \ref{coro3.4} is
%an immediately consequence of theorem \ref{teo3.1}, because in
%this case the vector of measures
%$\overrightarrow{\mu}=(\mu^{1},\ldots,\mu^{t})\in
%(\mathcal{M}(\mathbb{R}^d))^{(t)}$ is given by
%$d\mu^{j}=\psi^{j}dx$, $j=1,2,\ldots,t$, where $\psi^{j}\in
%L^1(\mathbb{R}^d)$, for $j=1,2,\ldots,t$, and $dx$ is the Lebesgue
%measure in $\mathbb{R}^d$. Note that in this case we have:
%\begin{displaymath}
%\|\overrightarrow{\mu}\|_{(\mathcal{M}(\mathbb{R}^d))^{(t)}}=\|\Psi\|_{(L^1(\mathbb{R}^d))^{(t)}}\quad\textrm{where}\quad\Psi=(\psi^{1},\ldots,\psi^{t}).
%\end{displaymath}
%\end{proof}

%PROOF COROLLARY 3.5
%{\bf Proof of corollary \ref{coro3.5}.}
%\begin{proof} In the proof of this
%corollary, we let $t=1$, and $\mu=\mu^{1}=\delta_{0}$, i.e., $\mu$
%is the Dirac measure in $\mathbb{R}^d$ at zero. Clearly,
%$\|\mu\|=\|\delta_{0}\|=1$, and by theorem \ref{teo3.1} the
%conclusion follows.
%\end{proof}

%PROOF THEOREM 3.3
{\bf Proof of Theorem \ref{teo3.3}.}
\begin{proof} Let $f =\sum_{k\in\mathbb{Z}^d}C_{k}^{T}\Phi_{k}\in V^{p}(\Phi)$,
%be given. Using the representation
%$f$,
$C\in(\ell^{p}(\mathbb{Z}^d))^{(r)}$. We have
\begin{eqnarray}
A_{p}\|f\|_{L^{p}}&\leq&\|(f\ast\overrightarrow{\mu})(X)\|_{(\ell^{p}(J))^{(t)}}\nonumber\\
&\leq&\|(f\ast(\overrightarrow{\mu}-\overrightarrow{\alpha}))(X)\|_{(\ell^{p}(J))^{(t)}}+\|(f\ast\overrightarrow{\alpha})(X)\|_{(\ell^{p}(J))^{(t)}}\nonumber\\
&=&\sum_{l=1}^{t}\|(f\ast(\mu^{l}-\alpha^{l}))(X)\|_{\ell^{p}(J)}+\|(f\ast\overrightarrow{\alpha})(X)\|_{(\ell^{p}(J)^{(t)}}\nonumber\\
&=&\sum_{l=1}^{t}\|((\sum_{k\in\mathbb{Z}^d}C_{k}^{T}\Phi_{k})\ast(\mu^{l}-\alpha^{l}))(X)\|_{\ell^{p}(J)}+\|(f\ast\overrightarrow{\alpha})(X)\|_{(\ell^{p}(J))^{(t)}}.\nonumber
\end{eqnarray}
Since $\overrightarrow{\mu}$ and $\overrightarrow{\alpha}$ are in
$(\mathcal{M}(\mathbb{R}^d))^{(t)}$, and
$\Phi\in(\mathit{W}_{0}^{1})^{(r)}$, then Proposition
\ref{prop4.2} implies
$\Omega^{l}=(\phi^{1}\ast(\mu^{l}-\alpha^{l}),\ldots,\phi^{r}\ast(\mu^{l}-\alpha^{l}))\in(\mathit{W}_{0}^{1})^{(r)}$,
for $l=1,\ldots,t$. Using Proposition \ref{prop4.2} once again  we
have:
\begin{eqnarray}
A_{p}\|f\|_{L^{p}}&\leq&\sum_{l=1}^{t}\mathcal{N}\|C\|_{(\ell^{p}(\mathbb{Z}^d))^{(r)}}\|\Omega^{l}\|_{(\mathit{W}^{1})^{(r)}}+ \|(f\ast\overrightarrow{\alpha})(X)\|_{(\ell^{p}(J))^{(t)}}\nonumber\\
&\leq&2^{d}\mathcal{N}\|C\|_{(\ell^{p}(\mathbb{Z}^d))^{(r)}}\|\Phi\|_{(\mathit{W}^{1})^{(r)}}
\|\overrightarrow{\mu}-\overrightarrow{\alpha}\|_{(\mathcal{M}(\mathbb{R}^d))^{(t)}}+\|(f\ast\overrightarrow{\alpha})(X)\|_{(\ell^{p}(J))^{(t)}}.\nonumber
\end{eqnarray}
Taking into account $\Phi\in(\mathit{W}^{1})^{(r)}$ also satisfies
(\ref{2.1.1}), and $f$ satisfies (\ref{1.2}), then it follows
\begin{equation}
2^{d}\mathcal{N}\|C\|_{(\ell^{p}(\mathbb{Z}^d))^{(r)}}\|\Phi\|_{(\mathit{W}^{1})^{(r)}}\|\overrightarrow{\mu}-\overrightarrow{\alpha}\|_{(\mathcal{M}(\mathbb{R}^d))^{(t)}}\nonumber\\
\leq
\frac{2^{d}\mathcal{N}\|f\|_{L^{p}}\|\Phi\|_{(\mathit{W}^{1})^{(r)}}\|\overrightarrow{\mu}-\overrightarrow{\alpha}\|_{(\mathcal{M}(\mathbb{R}^d))^{(t)}}}{m_{p}}.\nonumber
\end{equation}

Hence,
\begin{equation}\label{3.3.2}
\left(A_{p}-\frac{2^{d}\mathcal{N}\|\Phi\|_{(\mathit{W}^{1})^{(r)}}
\|\overrightarrow{\mu}-\overrightarrow{\alpha}\|_{(\mathcal{M}(\mathbb{R}^d))^{(t)}}}{m_{p}}\right)\|f\|_{L^{p}}\leq
\|(f\ast\overrightarrow{\alpha})(X)\|_{(\ell^{p}(J))^{(t)}}.
\end{equation}
On the other hand, since $f \in V^{p}(\Phi)$ satisfies
(\ref{1.2}), we have
\begin{eqnarray}
\|(f\ast\overrightarrow{\alpha})(X)\|_{(\ell^{p}(J))^{(t)}}&\leq&\|(f\ast(\overrightarrow{\alpha}-\overrightarrow{\mu}))(X)\|_{(\ell^{p}(J))^{(t)}}+\|(f\ast\overrightarrow{\mu})(X)\|_{(\ell^{p}(J))^{(t)}}\nonumber\\
&=&\sum_{l=1}^{t}\|(f\ast(\alpha^{l}-\mu^{l}))(X)\|_{\ell^{p}(J)}+\|(f\ast\overrightarrow{\mu})(X)\|_{(\ell^{p}(J))^{(t)}}\nonumber\\
&\leq&\sum_{l=1}^{t}\|((\sum_{k\in\mathbb{Z}^d}C_{k}^{T}\Phi_{k})\ast(\alpha^{l}-\mu^{l}))(X)\|_{\ell^{p}(J)}+B_{p}\|f\|_{L^{p}}\nonumber\\
&\leq&2^{d}\mathcal{N}\|C\|_{(\ell^{p}(\mathbb{Z}^d))^{(r)}}\|\Phi\|_{(\mathit{W}^{1})^{(r)}}\|\overrightarrow{\mu}-\overrightarrow{\alpha}\|_{(\mathcal{M}(\mathbb{R}^d))^{(t)}}+B_{p}\|f\|_{L^{p}}.\nonumber
\end{eqnarray}
Using condition (\ref{2.1.1}), we obtain:
\begin{equation}\label{3.3.3}
\|(f\ast\overrightarrow{\alpha})(X)\|_{(\ell^{p}(J))^{(t)}}\leq\left(\frac{2^{d}\mathcal{N}\|\Phi\|_{(\mathit{W}^{1})^{(r)}}\|\overrightarrow{\mu}-\overrightarrow{\alpha}\|_{(\mathcal{M}(\mathbb{R}^d))^{(t)}}}{m_{p}}+B_{p}\right)\|f\|_{L^{p}}.
\end{equation}
From (\ref{3.3.2}) and (\ref{3.3.3}), by choosing
\begin{displaymath}
\epsilon_{0}=\frac{A_{p}m_{p}}{2^{d}\mathcal{N}\|\Phi\|_{(\mathit{W}^{1})^{(r)}}},
\end{displaymath}
we obtain for
$0<\epsilon<\epsilon_{0}$,
\begin{displaymath}
A_{p}^{'}=A_{p}-\frac{2^{d}\mathcal{N}\|\Phi\|_{(\mathit{W}^{1})^{(r)}}}{m_{p}}\epsilon,\quad\textrm{and}
\end{displaymath}
\begin{displaymath}
B_{p}^{'}=B_{p}+\frac{2^{d}\mathcal{N}\|\Phi\|_{(\mathit{W}^{1})^{(r)}}}{m_{p}}\epsilon.
\end{displaymath}
\end{proof}

%%PROOF OF COROLLARY 3.6
%{\bf Proof of corollary \ref{coro3.6}.}
%\begin{proof} Let
%$\Psi\in(L^1(\mathbb{R}^d))^{(t)}$ be given, and define the
%measure by $d\overrightarrow{\mu}=\Psi dx$, where $dx$ is the
%Lebesgue measure in $\mathbb{R}^d$. Using the hypothesis and the
%fact $\Psi$ satisfies (\ref{3.6.1}) %(which is equivalent to
%%(\ref{1.2}) %in theorem \ref{teo3.3}),
%then by theorem
%\ref{teo3.3} the thesis of corollary \ref{coro3.6} follows if we
%define the measure $d\overrightarrow{\alpha}=\Gamma dx$, where
%$\Gamma\in(L^1(\mathbb{R}^d))^{(t)}$, and taking into account
%$\|\Psi-\Gamma\|_{(L^1(\mathbb{R}^d))^{(t)}}=\|\overrightarrow{\mu}-\overrightarrow{\alpha}\|_{(\mathcal{M}(\mathbb{R}^d))^{(t)}}.$
%\end{proof}

%PROOF OF THEOREM 3.4.
{\bf Proof of theorem \ref{teo3.4}.}

The conclusion of the theorem is essentially obvious at this point.
We proceed with a formal proof in order to obtain estimates for
$\epsilon_0$ and the bounds $A^{'}_{p}$ and $B^{'}_{p}$ of $X$ as an
$\overrightarrow{\alpha}$-sampling set for $V^p(\Theta)$.

\begin{proof} Let
$0<\epsilon_{1}<\frac{1}{2}\left(\sqrt{C_{p}^{2}+\frac{4A_{p}m_{p}^{2}}{2^{d}\mathcal{N}\|\overrightarrow{\mu}\|_{(\mathcal{M}(\mathbb{R}^d))^{(t)}}}}-C_{p}\right)$,
where
\begin{displaymath}
C_{p}=\|\Phi\|_{(\mathit{W}^{1})^{(r)}}+\frac{A_{p}m_{p}}{2^{d}\mathcal{N}\|\overrightarrow{\mu}\|_{(\mathcal{M}(\mathbb{R}^d))^{(t)}}}.
\end{displaymath}
Then, by Theorem \ref{teo3.1}, $X$ is a $\overrightarrow{\mu}$-sampling set for $V^{p}(\Theta)$ %with respect to $\overrightarrow{\mu}$,
as soon as
\begin{displaymath}
\|\Phi-\Theta\|_{(\mathit{W}^{1})^{(r)}}\leq\epsilon_{1}. %<\frac{1}{2}\left(\sqrt{C_{p}^{2}+\frac{4A_{p}m_{p}^{2}}{2^{d}\mathcal{N}}}-C_{p}\right),
\end{displaymath}
%$\mathcal{N}$ and $m_{p}$ are as above, and
Moreover,
\begin{displaymath}
A^{''}_{p}\|g\|_{L^{p}}\leq\|(g\ast\overrightarrow{\mu})(X)\|_{(\ell^{p}(J))^{(t)}}\leq
B^{''}_{p}\|g\|_{L^{p}},\mbox{ for all } g\in V^{p}(\Theta),
\end{displaymath}
where %, if $0<\epsilon<\epsilon_1$,
\begin{displaymath}
A^{''}_{p}=\frac{A_{p}m_{p}}{\|\Phi\|_{(\mathit{W}^{1})^{(r)}}+\epsilon_1}-\frac{2^{d}\mathcal{N}\|\overrightarrow{\mu}\|_{(\mathcal{M}(\mathbb{R}^d))^{(t)}}}{m_{p}-\epsilon_1}\epsilon_1
\end{displaymath}
and
\begin{displaymath}
B^{''}_{p}=\frac{2^{d}\mathcal{N}\|\overrightarrow{\mu}\|_{(\mathcal{M}(\mathbb{R}^d))^{(t)}}(\|\Phi\|_{(\mathit{W}^{1})^{(r)}}+\epsilon_1)}{m_{p}-\epsilon_1}.
\end{displaymath}
Assume now that
\begin{displaymath}
0<\epsilon_{2}\leq\frac{A_{p}^{''}(m_{p}-\epsilon_{1})}{2^{d}\mathcal{N}(\|\Phi\|_{(\mathit{W}^{1})^{(r)}}+\epsilon_{1})}.
\end{displaymath}
Then, by Theorem \ref{teo3.3}, $X$ is an $\overrightarrow{\alpha}$-sampling set for
$V^{p}(\Theta)$ as soon as
\[\|\Phi-\Theta\|_{(\mathit{W}^{1})^{(r)}} \leq \epsilon_1 \mbox{ and }
\norm{\overrightarrow{\mu}-\overrightarrow{\alpha}}_{(\mathcal{M}(\mathbb{R}^d))^{(t)}} \leq \epsilon_2.\]
%Moreover
%\begin{displaymath}
%\epsilon_{3}=\frac{1}{2}\left(\sqrt{\|\Phi\|_{(\mathit{W}^{1})^{(r)}}^{2}+\frac{4A_{p}^{''}m_{p}^{2}}{2^{d}\mathcal{N}}}-\|\Phi\|_{(\mathit{W}^{1})^{(r)}}\right).
%\end{displaymath}
%If we take
Hence, if $0<\epsilon<\epsilon_{0}=\min\{\epsilon_{1},\epsilon_{2}\}$, %,\epsilon_{3}\}$,
%then for $0<\epsilon<\epsilon_{0}$, we choose
we obtain the sampling bounds
\begin{displaymath}
A_{p}^{'}=A_{p}^{''}-\frac{2^{d}\mathcal{N}(\|\Phi\|_{(\mathit{W}^{1})^{(r)}}+\epsilon_1)}{m_{p}-\epsilon_1}\epsilon_2,
\end{displaymath}
and
\begin{displaymath}
B_{p}^{'}=B_{p}^{''}+\frac{2^{d}\mathcal{N}(\|\Phi\|_{(\mathit{W}^{1})^{(r)}}+\epsilon_1)}{m_{p}-\epsilon_1}\epsilon_2,
\end{displaymath}
as soon as
\[\|\Phi-\Theta\|_{(\mathit{W}^{1})^{(r)}} +
\norm{\overrightarrow{\mu}-\overrightarrow{\alpha}}_{(\mathcal{M}(\mathbb{R}^d))^{(t)}} \leq \epsilon < \epsilon_0.\]
%Moreover
\end{proof}

%PROOF OF COROLLARY 3.7.
%{\bf Proof of corollary \ref{coro3.7}.}
%\begin{proof}
%Let $\Psi\in(L^1(\mathbb{R}^d))^{(t)}$ and
%$\Phi\in(\mathit{W}_{0}^{1})^{(r)}$ be given. Define the measure
%$d\overrightarrow{\mu}=\Psi dx$, where $dx$ is the Lebesgue
%measure in $\mathbb{R}^d$. For any
%$\Gamma\in(L^1(\mathbb{R}^d))^{(t)}$ let $\overrightarrow{\alpha}$
%be defined by $d\overrightarrow{\alpha}=\Gamma dx$. Then the
%thesis of the corollary follows from theorem \ref{teo3.4} taking
%into account that
%$\|\Psi-\Gamma\|_{(L^1(\mathbb{R}^d))^{(t)}}=\|\overrightarrow{\mu}-\overrightarrow{\alpha}\|_{(\mathcal{M}(\mathbb{R}^d))^{(t)}}.$
%\end{proof}
{\bf Proof of Theorem \ref{teo5.1}.}
%\medskip

%As before, we let $\widetilde{X} = X+\Delta = \{x_j +\delta_j\}_{j\in J}$,
%where $\Delta=\{\delta_{j}\}_{j\in J}\subset\mathbb{R}^d$,
%and $U_{\Delta} = U_{(\widetilde{X},\Phi,\overrightarrow{\mu})}$ be the sampling operator
%that corresponds to $\widetilde{X}$.
%%where $\Delta=\{\delta_{j}\}_{j\in J}\subset\mathbb{R}^d$. Let
%%$U_{\Delta}$ be the linear operator on $(\ell^p(\mathbb{Z}^d))^{(r)}$
%Then $U_{\Delta}C=(f\ast\overrightarrow{\mu})(\widetilde{X})$, $C\in (\ell^p(\mathbb{Z}^d))^{(r)}$,
%where $f = \sum_{k\in\GG} C_k^T\Phi_k$.

The theorem is immediately implied by Lemma \ref{nutshell} and the following result.

%Lemma 5.1
\begin{lem}\label{lem5.1} Let $(X,\Phi,\overrightarrow{\mu})$ be a $p$-stable sampling model for some $p\in[1,\infty]$
and $\widetilde{X} = X+\Delta$. Let $U$ be the sampling operator for $(X,\Phi,\overrightarrow{\mu})$ and
$U_{\Delta}$  be the sampling operator for $(\widetilde{X},\Phi,\overrightarrow{\mu})$.
Then $\|U-U_{\Delta}\|\rightarrow0$ as $\|\Delta\|_{\infty}\rightarrow0$.
\end{lem}

\begin{proof}
We recall that for any $\gamma>0$, the function
$\operatorname{osc}_{\gamma}\,g$ on $\mathbb{R}^d$ is defined by
\begin{displaymath}
\operatorname{osc}_{\gamma}\,g(x)=\sup_{|\Delta
x|<\gamma}|g(x+\Delta x)-g(x)|.
\end{displaymath}
From Lemma 8.1 in \cite{aagr2} it follows that if
$g\in\mathit{W}_{0}^{1}$, then
$\operatorname{osc}_{\gamma}\,g\in\mathit{W}^1$, and
$\|\operatorname{osc}_{\gamma}\,g\|_{\mathit{W}^1}\rightarrow0$ as
$\gamma\rightarrow0$. Therefore, by applying Proposition
\ref{prop4.2} we get
\begin{displaymath}
\operatorname{osc}_{\gamma}\,\Phi\ast\overrightarrow{\mu}\in(\mathit{W}^1)^{(r\times
t)},\ \textrm{and}\ \|\operatorname{osc}_{\gamma}\,\Phi\ast\overrightarrow{\mu}\|_{(\mathit{W}^1)^{(r\times
t)}}\rightarrow0\ \textrm{as}\ \gamma\rightarrow0,
\end{displaymath}
where
\begin{displaymath}
\operatorname{osc}_{\gamma}\,\Phi\ast\overrightarrow{\mu}=\left(\begin{array}{ccc}
\operatorname{osc}_{\gamma}\,\phi^{1}\ast\mu^{1}&
\ldots&\operatorname{osc}_{\gamma}\,\phi^{1}\ast\mu^{t}\\
\vdots&&\vdots\\
\operatorname{osc}_{\gamma}\,\phi^{r}\ast\mu^{1}&\ldots&\operatorname{osc}_{\gamma}\,\phi^{r}\ast\mu^{t}
\end{array}\right).
\end{displaymath}
For any $m\in\mathbb{Z}^d$ there exist at most
$([\delta^{-1}\sqrt{d}]+1)^{d}$ sampling points in every hypercube %of volume 1
$[0,1]^{d}+m$. We set
$X_{m}=X\bigcap([0,1]^{d}+m)$, $m\in\mathbb{Z}^d$, and, for each $1\leq i\leq r$ and $1\leq
l\leq t$, define the sequence
\begin{displaymath}
b^{i,l}(m):=\operatorname{esssup}_{x\in[0,1]^{d}}\,\{\operatorname{osc}_{\|\Delta\|_{\infty}}\,(\phi^{i}\ast\mu^{l})(x+m)\},\quad
m\in\mathbb{Z}^d.
\end{displaymath}
Then
$\|b^{i,l}\|_{\ell^1(\mathbb{Z}^d)}=\|\operatorname{osc}_{\|\Delta\|_{\infty}}\,(\phi^{i}\ast\mu^{l})\|_{\mathit{W}^{1}}$ and,
hence, \[\|b\|_{(\ell^1(\mathbb{Z}^d))^{(r \times
t)}}=\|\operatorname{osc}_{\|\Delta\|_{\infty}}\,\Phi\ast\overrightarrow{\mu}\|_{(\mathit{W}^{1})^{(r
\times t)}}.\]
 For $1\leq i\leq r$ and $1\leq l\leq t$ we have
\begin{eqnarray}
\|(U^{i,l}-U_{\Delta}^{i,l})c^{i}\|_{\ell^p(J)}^{p}&=&\sum_{x_{j}\in
X}\left|\sum_{k\in\mathbb{Z}^d}c_{k}^{i}\left((\phi^{i}\ast\mu^{l})(x_{j}-k))-(\phi^{i}\ast\mu^{l})(x_{j}+\delta_{j}-k)\right)\right|^{p}\nonumber\\
&\leq&\sum_{x_{j}\in X}\left(\sum_{k\in\mathbb{Z}^d}|c_{k}^{i}|\operatorname{osc}_{\|\Delta\|_{\infty}}\,(\phi^{i}\ast\mu^{l})(x_{j}-k)\right)^{p}\nonumber\\
&\leq&\sum_{m\in\mathbb{Z}^d}\mathcal{N}^{p}\left(\sum_{k\in\mathbb{Z}^d}|c_{k}^{i}|b^{i,l}(m-k)\right)^{p}\nonumber\\
&=&\mathcal{N}^{p}\||c^{i}|\ast
b^{i,l}\|_{\ell^p(\mathbb{Z}^d)}^{p},\nonumber
\end{eqnarray}
where $\mathcal{N}=(\delta^{-1}\sqrt{d}+1)^{d/p}$. By using
Young's inequality we obtain
\begin{eqnarray}
\mathcal{N}^{p}\||c^{i}|\ast b^{i,l}\|_{\ell^p(\mathbb{Z}^d)}^{p}&\leq&\mathcal{N}^{p}\|c^{i}\|_{\ell^p(\mathbb{Z}^d)}^{p}\|b^{i,l}\|_{l^1}^{p}\nonumber\\
&=&\mathcal{N}^{p}\|c^{i}\|_{\ell^p(\mathbb{Z}^d)}^{p}\|\operatorname{osc}_{\|\Delta\|_{\infty}}\,\phi^{i}\ast\mu^{l}\|_{\mathit{W}^{1}}^{p}.\nonumber
\end{eqnarray}
Consequently,
\begin{displaymath}
\|U^{i,l}-U_{\Delta}^{i,l}\|\leq
\mathcal{N}\|\operatorname{osc}_{\|\Delta\|_{\infty}}\,\phi^{i}\ast\mu^{l}\|_{\mathit{W}^{1}}.
\end{displaymath}
Hence,
\begin{displaymath}
\|U-U_{\Delta}\|\leq
\mathcal{N}\|\operatorname{osc}_{\|\Delta\|_{\infty}}\,\Phi\ast\overrightarrow{\mu}\|_{(\mathit{W}^{1})^{(r
\times t)}} \rightarrow0\quad\textrm{as}\quad\|\Delta\|_{\infty}
\rightarrow0,
\end{displaymath}
and the lemma is proved.
\end{proof}

{\bf Proof of Theorem \ref{bigshell}.}

\bpf
The proof of Theorem \ref{bigshell} is hidden in the proofs of Theorems \ref{teo3.1}, \ref{teo3.3}, and \ref{teo5.1}.
In particular, keeping the notation of the proof of Theorem \ref{teo3.1}, we have
\[\norm{((f-g)\ast\overrightarrow{\mu})(X)}_{(\lp(J))^{(t)}} \leq
2^{d}\mathcal{N}\|C\|_{(\ell^{p}(\mathbb{Z}^d))^{(r)}}\|\Phi-\Theta\|_{(\mathit{W}^{1})^{(r)}}\|\overrightarrow{\mu}\|_{(\mathcal{M}(\mathbb{R}^d))^{(t)}}.
\]
Hence, Theorem \ref{bigshell} is true, when $\overrightarrow{\mu} = \overrightarrow{\alpha}$ and $X = X +\Delta$.
Keeping the notation of the proof of Theorem \ref{teo3.3}, we have
\[
\|(f\ast(\overrightarrow{\alpha}-\overrightarrow{\mu}))(X)\|_{(\ell^{p}(J))^{(t)}}
\leq 2^{d}\mathcal{N}\|C\|_{(\ell^{p}(\mathbb{Z}^d))^{(r)}}\|\Phi\|_{(\mathit{W}^{1})^{(r)}}\|\overrightarrow{\mu}-\overrightarrow{\alpha}\|_{(\mathcal{M}(\mathbb{R}^d))^{(t)}}.
\]
This inequality implies Theorem \ref{bigshell} when $\Phi = \Theta$ and $X = X+\Delta$.
Combining these results with  Theorem \ref{teo5.1} via the standard $\epsilon/3$ argument we prove the general case.
\epf

\subsection{Proofs for Section 3.2.}\

We begin with an auxiliary technical  result for the convolution of functions with measures.

%Lemma4.4
\begin{lem}\label{lem4.4} Let $\Phi=(\phi^{1},\ldots,\phi^{r})^{T}$ be a vector of continuous functions, $s > d$, and
$\overrightarrow{\mu}\in(\mathcal{M}_{s}(\mathbb{R}^d))^{(t)}$. % be given.
If
%\begin{displaymath}
$|\phi^{i}(x)|\leq C_{0}^{i}(1+|x|)^{-s}$ for all $1\leq i\leq r$,
%\end{displaymath}
then
\begin{displaymath}
|(\Phi\ast\overrightarrow{\mu})(x)|\leq C_{1}(1+|x|)^{-s};
\end{displaymath}
%where
the constants $C_{0}^{i}>0$, $1\leq i\leq r$, and $C_{1}>0$ are independent of
$x\in\mathbb{R}^d$.
\end{lem}
%Proof of Lemma4.4
\begin{proof}
For $1\leq i\leq r$ and $1\leq j\leq t$ we have
\begin{eqnarray}
|(\phi^{i}\ast\mu^{j})(x)|&\leq&\int_{\mathbb{R}^d}|\phi^{i}(x-y)|d|\mu^{j}|(y)\nonumber \\
&\leq&C_{0}^{i}\int_{\mathbb{R}^d}(1+|x-y|)^{-s}d|\mu^{j}|(y).\nonumber\
\end{eqnarray}
Since $(1+|u+w|)^{-l}\leq
(1+|u|)^{l}(1+|w|)^{-l}$, for all $u, w\in\mathbb{R}^d$, and
$l\geq0$, %then by choosing $u=-y$, $w=x$, and $l=r_{0}$ in the above inequality,
we have
\begin{eqnarray}
|(\phi^{i}\ast\mu^{j})(x)|&\leq&C_{0}^{i}\int_{\mathbb{R}^d}(1+|y|)^{s}(1+|x|)^{-s}d|\mu^{j}|(y)\nonumber\\
&=&C_{0}^{i}(1+|x|)^{-s}\int_{\mathbb{R}^d}(1+|y|)^{s}d|\mu^{j}|(y)\nonumber\\
&\leq& C_{1}^{i,j}(1+|x|)^{-s}\nonumber,
\end{eqnarray}
where the last inequality follows from $\mu^j\in\mathcal M_s(\RR^d)$.
%Since $S_{\mu^{s}}$ is compact, where $S_{\mu^{s}}$ denotes the
%support of the measure $\mu^{s}$, then so is $S_{|\mu^{s}|}$. Hence,
%\begin{eqnarray}
%|(\phi^{i}\ast\mu^{s})(x)|&\leq&C_{0}^{i}(1+|x|)^{-r_{0}}\int_{S_{|\mu^{s}|}}(1+|y|)^{r_{0}}d|\mu^{s}|(y)\nonumber\\
%&\leq&C_{1}^{i,s}(1+|x|)^{-r_{0}},\nonumber\
%\end{eqnarray}
%where $C_{1}^{i,s}:=C_{0}^{i}(\max_{y\in
%S_{|\mu^{s}|}}(1+|y|)^{r_{0}})\|\mu^{s}\|$.
Therefore,
\begin{displaymath}
|(\Phi\ast\overrightarrow{\mu})(x)|\leq C_{1}(1+|x|)^{-s},
\end{displaymath}
where $C_{1}=\sum_{i=1}^{r}\sum_{j=1}^{t}C_{1}^{i,j}$.
\end{proof}

\begin{obs}\label{obs3.4}
If $\{\Phi_{k}\}_{k\in\mathbb{Z}^d}$ is an $s$-localized Riesz generator for $V^2(\Phi)$, as in Definition \ref{sRiesz},
% and the components $\phi^{i}$, $1\leq i\leq r$, of
%$\Phi$ satisfy (\ref{3.8.3}),
then, by Lemma 14(a) in \cite{kg1}, we have that
$\{\widetilde{\Phi}_{k}\}_{k\in\mathbb{Z}^d}$ is also an $s$-localized Riesz generator for $V^2(\Phi)$.
Consequently, by Lemma \ref{lem4.4} we have
\begin{equation}\label{3.8.5}
|(\widetilde{\Phi}\ast\overrightarrow{\mu})(x)|\leq
D_{1}(1+|x|)^{-s},
\end{equation}
for some $D_{1}>0$ independent of $x\in\mathbb{R}^d$.
\end{obs}

%The next proposition shows that a frame algorithm can be used to reconstruct a function from its samples.

%%Proposition4.3
%\begin{prop}\label{prop4.3} Let $\Phi\in(\mathit{W}_{0}^{1})^{(r)}$, $\overrightarrow{\mu}\in(\mathcal{M}(\mathbb{R}^d))^{(t)}$, and $X$ be
% a $\overrightarrow{\mu}$-sampling set for $V^2(\Phi)$.
% Then
%there exists a sequence of vectors of functions
%$\{\Psi_{x_{j}}\}_{j\in J}$, which is a frame for $V^2(\Phi)$
%and $\langle
%f,\Psi_{x_{j}}\rangle=(f\ast\overrightarrow{\mu})(x_{j})$ for all
%$f\in V^2(\Phi)$ and $j\in J$. Moreover, every function $f\in
%V^2(\Phi)$ can be recovered from the sequence of its samples
%$\{(f\ast\overrightarrow{\mu})(x_{j})\}_{j\in J}$ via
%\begin{equation}\label{3.8.5.1}
%f(x)=\sum_{j\in J}(f\ast\overrightarrow{\mu})(x_{j})\widetilde{\Psi}_{x_{j}}(x),
%\end{equation}
%where $\{\widetilde{\Psi}_{x_{j}}\}_{j\in J}$ is a dual frame of
%$\{\Psi_{x_{j}}\}_{j\in J}$ and the series (\ref{3.8.5.1})
%converges unconditionally in $V^2(\Phi)$.
%\end{prop}

{\bf Proof of Proposition \ref{prop4.3}.}

%Proof of Proposition4.3.
\begin{proof}
Let $X$ be a $\overrightarrow{\mu}$-sampling set for  $V^2(\Phi)$,
$\overrightarrow{\mu}\in(\mathcal{M}(\mathbb{R}^d))^{(t)}$.
Then, by definition, there exist constants $0<A_{2}\leq B_{2}<\infty$
such that
\begeq\label{s2}
A_{2}\|f\|_{L^{2}}\leq \|(f\ast\overrightarrow{\mu})(X)\|_{(\ell^{2}(J))^{(t)}}\leq
B_{2}\|f\|_{L^{2}},\mbox{for all } f\in V^{2}(\Phi).
\end{equation}
Fix $x_{j}\in X$. Then, for each $1\leq i\leq t$, the
function $g_{x_{j}}^i$: $V^2(\Phi)\rightarrow\mathbb{C}$ given by
$g_{x_{j}}^i(f)=(f\ast\mu^i)(x_{j})$ is a bounded linear functional
on the closed subspace $V^2(\Phi)$ of $L^2(\mathbb{R}^d)$ because
$|g_{x_{j}}^i(f)|\leq B_{2}\|f\|_{L^{2}}$ for all $f\in V^2(\Phi)$.
Consequently, by Riesz representation theorem, there exists
$\psi_{x_{j}}^i\in V^2(\Phi)$ such that $g_{x_{j}}^i(f)=\langle
f,\psi_{x_{j}}^i\rangle$ for all $f\in V^2(\Phi)$. It follows immediately from \eqref{s2} and Definition \ref{def3.1} that
$\Psi_{x_{j}}=(\psi_{x_{j}}^1,\ldots,\psi_{x_{j}}^t)^{T}$ is a frame for $V^2(\Phi)$.
%, because in this
%case we have that $s_{x_{j}}(f)=\langle f,\Psi_{x_{j}}\rangle$, for
%all $f\in V^2(\Phi)$, and for all $j\in J$.
%The proof of
%(\ref{3.8.5.1}) follows from the fact that $\{\Psi_{x_{j}}\}_{j\in
%J}$ is a frame for $V^2(\Phi)$, and thus
Hence, every $f\in V^2(\Phi)$ can
be recovered via $f=\sum_{j\in J}\langle
f,\Psi_{x_{j}}\rangle\widetilde{\Psi}_{x_{j}}$, where
$\{\widetilde{\Psi}_{x_{j}}=(\widetilde{\psi}_{x_{j}}^{1},\ldots,\widetilde{\psi}_{x_{j}}^{t})^{T}\}_{j\in
J}$ is a dual frame of $\{\Psi_{x_{j}}\}_{j\in J}$ and the series
converges unconditionally in $V^2(\Phi)$. Since
$\langle f,\Psi_{x_{j}}\rangle=(f\ast\overrightarrow{\mu})(x_{j})$
for all $j\in J$, we get (\ref{3.8.5.1}).
\end{proof}

Next, we show that if the generator $\Phi$ and the measures $\overrightarrow{\mu}$ satisfy
 an appropriate decay condition then the $(\overrightarrow{\mu},X)$-sampling frame $\{\Psi_{x_j}\}$
obtained above is $s$-localized.

%Proposition4.4
\begin{prop}\label{prop4.4} Let $s>d$, $\Phi\in\mathcal{W}_{s}$, and
$\overrightarrow{\mu}\in(\mathcal{M}_{s}(\mathbb{R}^d))^{(t)}$.
If $X$ is a $\overrightarrow{\mu}$-sampling set for $V^2(\Phi)$, then the $(\overrightarrow{\mu},X)$-sampling frame $\{\Psi_{x_j}\}$ is
$s$-localized with respect to the Riesz basis
$\{\Phi_{k}\}_{k\in\mathbb{Z}^d}$.
\end{prop}
%Proof of Proposition4.4
\begin{proof}
%Notice that our assumptions imply
%$\Phi\in(\mathit{W}_{0}^{1})^{(r)}$. Moreover, since
%$(\mathcal{M}_{s}(\mathbb{R}^d))^{(t)}\subset (\mathcal{M}(\mathbb{R}^d))^{(t)}$
% and $X$ is a $\overrightarrow{\mu}$-sampling set for
%$V^2(\Phi)$, then, by Proposition
%\ref{prop4.3}, there exists a frame $\{\Psi_{x_{j}}\}_{j\in J}$ for
%$V^2(\Phi)$ such that
%\begin{displaymath}
%\langle f,\Psi_{x_{j}}\rangle=(f\ast\overrightarrow{\mu})(x_{j}),
%\mbox{ for all } f\in V^2(\Phi).
%\end{displaymath}
Since $\{\Phi_{k}\}_{k\in\mathbb{Z}^d}$ is an $s$-localized Riesz generator for $V^2(\Phi)$, the components of $\Phi$ satisfy
(\ref{3.8.3}), %for $r_{0}=s+d+\nu_{0}$, then by
and Lemma \ref{lem4.4} implies
\begin{displaymath}
|\langle
\Phi_{k},\Psi_{x_{j}}^{T}\rangle|=|(\Phi\ast\overrightarrow{\mu})(x_{j}-k)|\leq
C_{1}(1+|x_{j}-k|)^{-s},
\end{displaymath}
for some $C_{1}>0$ independent of $j\in J$ and $k\in\mathbb{Z}^d$.
On the other hand, %since $\{\Phi_{k}\}_{k\in\mathbb{Z}^d}$ is a
%Riesz basis for $V^2(\Phi)$, %and the components of $\Phi$ satisfy (\ref{3.8.3}), then by
 it follows from Remark \ref{obs3.4}
that the dual Riesz basis $\{\widetilde{\Phi}_{k}\}_{k\in\mathbb{Z}^d}$ is also an $s$-localized Riesz generator for $V^2(\Phi)$, and its
components %$\widetilde{\phi}^{i}$, $1\leq i\leq r$,
 %of  generator $\widetilde{\Phi}$ of $\Phi$
 also satisfy (\ref{3.8.3}).
%for $r_{0}=s+d+\nu_{0}$.
Therefore, using  Lemma \ref{lem4.4} once again, we get
\begin{displaymath}
|\langle\widetilde{\Phi}_{k},\Psi_{x_{j}}^{T}\rangle|=|(\widetilde{\Phi}\ast\overrightarrow{\mu})(x_{j}-k)|\leq
D_{1}(1+|x_{j}-k|)^{-s},
\end{displaymath}
for some $D_{1}>0$ independent of $j\in J$ and $k\in\mathbb{Z}^d$.
Hence, $\{\Psi_{x_j}\}$ satisfies all conditions of Definition \ref{def3.2}.
\end{proof}

We conclude this subsection with the proof of the main result of section 3.2.
\medskip

%Proof of Theorem3.5
{\bf Proof of Theorem \ref{teo3.5}}
\begin{proof}
Assume the hypotheses of Theorem \ref{teo3.5}.
By Propositions \ref{prop4.3} and \ref{prop4.4}, there exists a $(\overrightarrow{\mu},X)$-sampling frame
$\{\Psi_{x_{j}}\}_{j\in J}$ for $V^2(\Phi)$, which is $s$-localized with
respect to the Riesz basis $\{\Phi_{k}\}_{k\in\mathbb{Z}^d}$  and satisfies
\begin{displaymath}
\langle f,\Psi_{x_{j}}\rangle=(f\ast\overrightarrow{\mu})(x_{j}),
\mbox{ for all } f\in V^2(\Phi).
\end{displaymath}
Moreover,
\begin{displaymath}
f=\sum_{j\in
J}(f\ast\overrightarrow{\mu})(x_{j})\widetilde{\Psi}_{x_{j}},
\mbox{ for all } f\in V^2(\Phi).
\end{displaymath}
Consequently, applying Theorem 10(c) in \cite{kg1}, we get
\begin{displaymath}
f=\sum_{j\in
J}(f\ast\overrightarrow{\mu})(x_{j})\widetilde{\Psi}_{x_{j}},
\mbox{ for all } f\in V^p(\Phi),
\end{displaymath}
where the series converges unconditionally in $V^p(\Phi)$, $1\leq p<\infty$.
Moreover, since $\{\Psi_{x_{j}}\}_{j\in J}$ is an
$s$-localized frame with respect to the Riesz basis
$\{\Phi_{k}\}_{k\in\mathbb{Z}^d}$, then Theorem 10(d) in
\cite{kg1} implies that  for each $1\leq p\leq\infty$ there exist
$0<A_{p}\leq B_{p}<\infty$ such that
\begin{displaymath}
A_{p}\|f\|_{L^p}\leq\|(f\ast\overrightarrow{\mu})(X)\|_{(\ell^p(J))^{(t)}}\leq
B_{p}\|f\|_{L^p}, \mbox{ for all } f\in V^p(\Phi),
\end{displaymath}
i.e., $X$ is a $\overrightarrow{\mu}$-sampling set for $V^p(\Phi)$ and
the theorem is proved.
\end{proof}

\subsection{Proofs for section 3.3.}\
\medskip

%%Lemma5.1
%\begin{lem}\label{lem5.1} Let $X$ be a sampling set for
%$V^{p}(\Phi)$and $\overrightarrow{\mu}$, and assume that
%(\ref{2.1.1}) holds. Let $0<\eta_{p}\leq \beta_{p}<\infty$
%satisfying (\ref{5.1}). If $\|U-U_{\Delta}\|<\eta_{p}$, then
%$X+\Delta$ is also a set of sampling for $V^{p}(\Phi)$ and
%$\overrightarrow{\mu}$.
%\end{lem}
%\begin{proof}
%Let $C\in(\ell^p(\mathbb{Z}^d))^{(r)}$ be given. Then
%\begin{eqnarray}
%\|U_{\Delta}C\|_{(\ell^p(J))^{(t)}}&\leq&\|(U-U_{\Delta})C\|_{(\ell^p(J))^{(t)}}+\|UC\|_{(\ell^p(J))^{(t)}}\nonumber\\
%&\leq&\|U-U_{\Delta}\|\|C\|_{(\ell^p(\mathbb{Z}^d))^{(r)}}+\beta_{p}\|C\|_{(\ell^p(\mathbb{Z}^d))^{(r)}}.\nonumber
%\end{eqnarray}
%Therefore, since $\|U-U_{\Delta}\|<\eta_{p}$, then we have
%\begin{equation}\label{5.1.1}
%\|U_{\Delta}C\|_{(\ell^p(J))^{(t)}}\leq\left(\eta_{p}+\beta_{p}\right)\|C\|_{(\ell^p(\mathbb{Z}^d))^{(r)}}.
%\end{equation}
%On the other hand, since
%\begin{eqnarray}
%\eta_{p}\|C\|_{(\ell^p(\mathbb{Z}^d))^{(r)}}&\leq&\|UC\|_{(\ell^p(J))^{(t)}}\leq\|(U-U_{\Delta})C\|_{(\ell^p(J))^{(t)}}+\|U_{\Delta}C\|_{(\ell^p(J))^{(t)}}\nonumber\\
%&\leq&\|U-U_{\Delta}\|\|C\|_{(\ell^p(\mathbb{Z}^d))^{(r)}}+\|U_{\Delta}C\|_{(\ell^p(J))^{(t)}}.\nonumber
%\end{eqnarray}
%Hence
%\begin{equation}\label{5.1.2}
%\left(\eta_{p}-\|U-U_{\Delta}\|\right)\|C\|_{(\ell^p(\mathbb{Z}^d))^{(r)}}\leq\|U_{\Delta}C\|_{(\ell^p(J))^{(t)}}.
%\end{equation}
%Therefore, if $\|U-U_{\Delta}\|<\eta_{p}$, then from (\ref{5.1.1}),
%(\ref{5.1.2}), and proposition \ref{prop5.1}, the conclusion of the
%lemma follows.
%\end{proof}

For the proof of Theorem \ref{teo5.2} we need the following two lemmas.

%Lemma 5.2
\begin{lem}\label{lem5.2} %Let $\epsilon>0$, and $X$ a sampling set for $V^{2}(\Phi)$ and $\overrightarrow{\mu}$ be given.
%Assume that (\ref{2.1.1}) takes place, and there exists
%$\gamma_{0}>0$ such that $\|U-U_{\Delta}\|<\epsilon$ whenever
%$\|\Delta\|_{\infty}\leq \gamma_{0}$.
Let the assumptions of Theorem \ref{teo5.2} hold.
Then
\[\|U^{*}U-U_{\Delta}^{*}U_{\Delta}\|<\epsilon\left(2\beta_p+\epsilon\right).\]
%whenever $\|\Delta\|_{\infty}\leq \gamma_{0}$.
\end{lem}
\begin{proof}
%Let $\epsilon>0$ be given, and assume $\|\Delta\|_{\infty}\leq \gamma_{0}$.
Since $\|U\|=\|U^{*}\|$ and $\|U-U_{\Delta}\|=\|U^{*}-U_{\Delta}^{*}\|$, %then
\begin{eqnarray}
\|U^{*}U-U_{\Delta}^{*}U_{\Delta}\|&=&\|U^{*}U-U^{*}U_{\Delta}+U^{*}U_{\Delta}-U_{\Delta}^{*}U_{\Delta}\|\nonumber\\
&=&\|U^{*}(U-U_{\Delta})+(U^{*}-U_{\Delta}^{*})U_{\Delta}\|\nonumber\\
&\leq&\|U^{*}\|\|U-U_{\Delta}\|+\|U^{*}-U_{\Delta}^{*}\|\|U_{\Delta}\|\nonumber\\
&\leq&\|U-U_{\Delta}\|\left(\|U\|+\|U_{\Delta}\|\right)\nonumber\\
&\leq&\|U-U_{\Delta}\|\left(2\|U\|+\|U-U_{\Delta}\|\right)\nonumber\\
&\leq&\epsilon\left(2\beta_p+\epsilon\right),\nonumber
\end{eqnarray}
and the lemma is proved.
\end{proof}
%Lemma 5.3
\begin{lem}\label{lem5.3}
Let the assumptions of Theorem \ref{teo5.2} hold.
%Let
%$0<\epsilon<-\beta+\sqrt{\beta^{2}+\eta^{2}}$, where $0<\eta\leq
%\beta<\infty$ are constants satisfying (\ref{5.1}) when $p=2$.
%Assume there exists $\gamma_{0}>0$ such that
%$\|\Delta\|_{\infty}<\gamma_{0}$, and define
%$\nu=\nu(\epsilon)=\eta^{-2}\epsilon\left(\epsilon+2\beta\right)$.
Then $0<\nu<1$, $(U_{\Delta}^{*}U_{\Delta})^{-1}$ exists, and
$\|(U^{*}U)^{-1}-(U_{\Delta}^{*}U_{\Delta})^{-1}\|<\frac{\nu}{\eta_p^{2}\left(1-\nu\right)}$.
\end{lem}
\begin{proof}
Since $(U^{*}U)^{-1}$ exists, %because $X$ is a set of sampling for $V^{2}(\Phi)$ and $\overrightarrow{\mu}$. Then
\begin{equation}\label{5.3.1}
U_{\Delta}^{*}U_{\Delta}=U^{*}U\left(I+(U^{*}U)^{-1}\left(U_{\Delta}^{*}U_{\Delta}-U^{*}U\right)\right).
\end{equation}
%Notice that
From \eqref{5.1} and \eqref{5.1.1} we get that  for all $C\in(\lp(\mathbb{Z}^d))^{(r)}$
\begin{displaymath}
\frac{1}{\beta_p^{2}}
\|C\|_{(\lp(\mathbb{Z}^d))^{(r)}}\leq\|(U^{*}U)^{-1}C\|_{(\lp(\mathbb{Z}^d))^{(r)}}\leq\frac{1}{\eta_p^{2}}
\|C\|_{(\lp(\mathbb{Z}^d))^{(r)}}.
\end{displaymath}
From the above inequalities and Lemma \ref{lem5.2} we have
\begin{eqnarray}
&\|(U^{*}U)^{-1}\left(U_{\Delta}^{*}U_{\Delta}-U^{*}U\right)\| \leq\|(U^{*}U)^{-1}\|\|U_{\Delta}^{*}U_{\Delta}-U^{*}U\|\nonumber\\
&\leq \frac{1}{\eta_p^{2}} \epsilon\left(2\beta_p+\epsilon\right) <
\frac{1}{\eta_p^{2}} (-\beta_p+\sqrt{\beta_p^{2}+\eta_p^{2}})\left(2\beta_p-\beta_p+\sqrt{\beta_p^{2}+\eta_p^{2}}\right) =1. \nonumber
\end{eqnarray}
Hence, $\nu = \frac{1}{\eta_p^{2}} \epsilon\left(2\beta_p+\epsilon\right)\in(0,1)$. To simplify the notation, we define
\begin{displaymath}
M:=U^{*}U,\quad
M_{\Delta}:=U_{\Delta}^{*}U_{\Delta},\ \textrm{and}\
N:=(U^{*}U)^{-1}\left(U_{\Delta}^{*}U_{\Delta}-U^{*}U\right).
\end{displaymath}
Since $\|N\|\leq\nu<1$, then $(I+N)^{-1}$ exists and is given by the
Neumann series
\begin{displaymath}
(I+N)^{-1}=\sum_{q=0}^{\infty}(-1)^{q}N^{q}.
\end{displaymath}
From (\ref{5.3.1}) we obtain
\begin{equation}\label{5.3.2}
M_{\Delta}^{-1}=\left[M(I+N)\right]^{-1}=(I+N)^{-1}M^{-1}.
\end{equation}
Therefore, $M_{\Delta}^{-1}=(U_{\Delta}^{*}U_{\Delta})^{-1}$ exists.
%exists whenever $\|\Delta\|_{\infty}<\gamma_{0}$.

Now we need to give an upper bound for $\|M^{-1}-M_{\Delta}^{-1}\|$. %Assume $\|\Delta\|_{\infty}<\gamma_{0}$.
Using (\ref{5.3.2}) we obtain
\begin{displaymath}
M^{-1}-M_{\Delta}^{-1}=N(I+N)^{-1}M^{-1}.
\end{displaymath}
Consequently,
\begin{eqnarray}
\|M^{-1}-M_{\Delta}^{-1}\|&\leq&\|N\|\|(I+N)^{-1}\|\|M^{-1}\|\nonumber\\
&\leq&\frac{\|N\|}{1-\|N\|}\|M^{-1}\|\leq
\frac{\nu}{1-\nu}\eta_p^{-2},\nonumber\\
\end{eqnarray}
and the lemma is proved.
\end{proof}

{\bf Proof of theorem \ref{teo5.2}.}

\begin{proof}
%Let $\|\Delta\|_{\infty}<\gamma_{0}$.
Using the notations from Lemmas \ref{lem5.2}, \ref{lem5.3}, and the previous proofs, we get
\begin{eqnarray}
\|(U^{*}U)^{-1}U^{*}-(U_{\Delta}^{*}U_{\Delta})^{-1}U_{\Delta}^{*}\|=\|M^{-1}U^{*}-M_{\Delta}^{-1}U_{\Delta}^{*}\|\nonumber\\
=\|M^{-1}U^{*}-M^{-1}U_{\Delta}^{*}+M^{-1}U_{\Delta}^{*}-M_{\Delta}^{-1}U_{\Delta}^{*}\|\nonumber\\
=\|M^{-1}(U^{*}-U_{\Delta}^{*})+(M^{-1}-M_{\Delta}^{-1})U_{\Delta}^{*}\|\nonumber\\
\leq\|M^{-1}\|\|U^{*}-U_{\Delta}^{*}\|+\|M^{-1}-M_{\Delta}^{-1}\|\|U_{\Delta}^{*}\|\nonumber\\
\leq\frac{1}{\eta_p^{2}}\left(\epsilon+\frac{\nu\left(\epsilon+\beta_p\right)}{1-\nu}\right).\nonumber
\end{eqnarray}
\end{proof}

%Proof of Theorem \lastbig
{\bf Proof of Theorem \ref{lastbig}.}
\begin{proof}
Let $U_\Delta$ be the sampling operator for a perturbed sampling model $(X+\Delta, \Theta, \overrightarrow{\alpha})$.
Let also  $C\in (\lp(\GG))^{(r)}$, $f=\sum_{k\in
\mathbb{Z}^d}C_{k}^{T}\Phi_{k}$, and $g=\sum_{k\in
\mathbb{Z}^d}C_{k}^{T}\Theta_{k}$.
Then
\[
\norm{R(g\ast\overrightarrow{\alpha})(X+\Delta)-f}_{L^p} \le M_p
\norm{((U^*U)^{-1}U^*U_\Delta C- C)}_{\ell^p}.
\]
It remains to apply Theorem \ref{teo5.2} to finish the proof.
\end{proof}

%Proof of corollary 5
%{\bf Proof of corollary \ref{coro5}.}
%\begin{proof}
%Let $f\in V^2(\Phi)$ be given by $f=\sum_{k\in
%\mathbb{Z}^d}C_{k}^{T}\Phi_{k}$. Since
%$\Phi\in(\mathit{W}_{0}^{1})^{(r)}$ satisfies (\ref{2.1.1}) when
%$p=2$, and $(f\ast\overrightarrow{\mu})(X+\Delta)=U_{\Delta}C$ then
%there exists a constant $M_{2}>0$ independent of $f$,
%$\overrightarrow{\mu}$, $X$, and $\Delta$ such that
%\begin{displaymath}
%\|R_{X}^{(\overrightarrow{\mu},
%\Phi)}(f\ast\overrightarrow{\mu})(X+\Delta)-f\|_{L^2}\leq
%M_{2}\|(U^{*}U)^{-1}U^{*}U_{\Delta}C-C\|_{(l^2(\mathbb{Z^d}))^{(r)}}.
%\end{displaymath}
%Taking into account that
%\begin{displaymath}
%(U^{*}U)^{-1}U^{*}U_{\Delta}C-C=(U^{*}U)^{-1}U^{*}[U_{\Delta}-U]C,
%\end{displaymath}
%then we obtain
%\begin{displaymath}
%\|U(U^{*}U)^{-1}U^{*}U_{\Delta}C-C\|_{(l^2(\mathbb{Z^d}))^{(r)}} \leq \|(U^{*}U)^{-1}U^{*}\|_{2,op}\|U_{\Delta}-U\|_{2,op}\|C\|_{(l^2(\mathbb{Z}^d))^{(r)}}.
%\end{displaymath}
%Now the conclusion of Corollary \ref{coro5} follows from Theorem
%\ref{teo5.1}.
%\end{proof}
%Proof of corollary 5.1

{\bf Proof of Theorem \ref{lastcombine}.}
\begin{proof}
Assume the hypotheses of Theorem \ref{lastcombine}. From Theorem \ref{teo3.5} we know that,
in this case, the sampling model $(X,\Phi,\overrightarrow{\mu})$ is $p$-stable for
every $p\in[1,\infty]$. Hence, in view of Theorem \ref{lastbig}, the only thing that we need to prove
is that the operator $U^*U$ is invertible for all $p\in[1,\infty]$ and not just for $p=2$.

Taking into account that for each $1\leq i\leq r$ and $1\leq l\leq
t$ the entries of the matrix of the operator $U^{i,l}$ satisfy
\begin{displaymath}
|(U^{i,l})_{j,k}|=|(\phi^{i}\ast\mu^{l})(x_{j}-k)|\leq
C_{1}(1+|x_{j}-k|)^{-s},
\end{displaymath}
for some $C_{1}>0$ independent of $j\in J$ and $k\in\mathbb{Z}^d$,
it follows from Lemma 3 in \cite{kg1} that the matrix of $U$ defines
a bounded linear operator from
$(\ell^p(\mathbb{Z}^d))^{(r)}\rightarrow(\ell^p(J))^{(t)}$ for all $1\leq
p\leq \infty$. Hence, $U^{*}$ is also well defined as a bounded
linear operator from
$(\ell^p(J))^{(t)}\rightarrow(\ell^p(\mathbb{Z}^d))^{(r)}$, and, therefore,
$U^{*}U
:(\ell^p(\mathbb{Z}^d))^{(r)}\rightarrow(\ell^p(\mathbb{Z}^d))^{(r)}$ is a
well defined and bounded operator for all $1\leq p\leq\infty$. On the
other hand, since the operator $U^{*}U$ is invertible
on $(\ell^2(\mathbb{Z}^d))^{(r)}$ and its components $(M^{i,l})_{j,k}$, $1\leq i\leq r$, $1\leq l\leq r$,
satisfy a decay condition
\begin{displaymath}
|(M^{i,l})_{j,k}|\leq C_{2}(1+|x_{j}-k|)^{-s},
\end{displaymath}
for some $C_{2}>0$ independent of $j\in J$ and $k\in\mathbb{Z}^d$,
then Jaffard's Lemma (see Theorem 5 in \cite{kg1}) implies
that
$(U^{*}U)^{-1}:(\lt(\mathbb{Z}^d))^{(r)}\rightarrow(\lt(\mathbb{Z}^d))^{(r)}$
is also a bounded linear operator defined by a matrix satisfying the same
off-diagonal decay condition as $U^{*}U$. Consequently,
using  Lemma 3 in \cite{kg1} once again, we get that the matrix of $(U^{*}U)^{-1}$
defines a bounded linear operator on
$(\ell^p(\mathbb{Z}^d))^{(r)}$ %\rightarrow(\ell^p(\mathbb{Z}^d))^{(r)}$
for all $1\leq p\leq\infty$.
The theorem is proved.
\epf

%__________________________________________________________________________________________________________________________
%References

\end{document}